\documentclass[11pt]{amsart}
\usepackage{amsmath, amsthm, amssymb,amscd}
\usepackage{graphicx}
\usepackage[all,cmtip]{xy}
\usepackage[english]{babel}
\newcommand{\C}{\mathbb{C}}

\newcommand{\F}{\mathcal{F}}

\newtheorem{thm}{Theorem}[section]
\newtheorem*{thm-nl}{Theorem}
\newtheorem*{prop-nl}{Proposition}

\newtheorem{lem}[thm]{Lemma}

\def\PP{{\textbf P}}
\def\OO{\mathcal{O}}

\def\G{\mathcal{G}}
\def\F{\mathcal{F}}
\def\ms{\overline{\textbf{M}}}

\def\L{\mathcal{L}}

\def\cM{\mathcal{M}}
\def\cR{\mathcal{R}}

\def\cZ{\mathcal{Z}}

\def\cX{\mathcal{X}}

\def\H{\mathcal{H}}
\def\Pic0{{\rm Pic}^0(X)}

\def\mm{\overline{\mathcal{M}}}

\newtheorem{cor}[thm]{Corollary}
\newtheorem*{cor-nl}{Corollary}

\newtheorem*{conjecture-nl}{Conjecture}
\newtheorem*{quest-nl}{Question}
\newtheorem*{quests-nl}{Questions}

\newtheorem{prop}[thm]{Proposition}

\theoremstyle{remark}

\newtheorem{remark}[thm]{Remark}

\title{{The generic Green-Lazarsfeld Secant Conjecture}}

\author[G. Farkas]{Gavril Farkas}
\address{Humboldt-Universit\"at zu Berlin, Institut f\"ur Mathematik,  Unter den Linden 6
\hfill \newline\texttt{}
 \indent 10099 Berlin, Germany} \email{{\tt farkas@math.hu-berlin.de}}

\author[M. Kemeny]{Michael Kemeny}

\address{Humboldt-Universit\"at zu Berlin, Institut f\"ur Mathematik,  Unter den Linden 6
\hfill \newline\texttt{}
 \indent 10099 Berlin, Germany} \email{{\tt michael.kemeny@gmail.com}}

\begin{document}
\maketitle
\begin{abstract}
Using lattice theory on special $K3$ surfaces, calculations on moduli stacks of pointed curves and Voisin's proof of  Green's Conjecture on syzygies of canonical curves, we prove the Prym-Green Conjecture on the naturality of the resolution of a general Prym-canonical curve of odd genus, as well as (many cases of) the Green-Lazarsfeld Secant Conjecture on syzygies of non-special line bundles on general curves.
\end{abstract}
\section{Introduction}

For a smooth curve $C$ of genus $g$ and a very ample line bundle $L\in \mbox{Pic}^d(C)$, the Koszul cohomology groups $K_{p,q}(C,L)$ of $p$-th syzygies of weight $q$ of the embedded curve $\phi_L:C\rightarrow \PP^r$ are obtained from a minimal free resolution of the graded $S:=\mbox{Sym } H^0(C,L)$-module
$$\Gamma_C(L):=\bigoplus_{n\geq 0} H^0\bigl(C, L^{\otimes n}\bigr).$$
The \emph{graded Betti diagram} of $(C,L)$ is the table obtained by placing in the $p$th column and $q$th row the graded Betti number $b_{p,q}:=\mbox{dim } K_{p,q} (C,L)$. The resolution is said to be \emph{natural} if at most one Betti number along each diagonal is non-zero, which amounts to the statement $b_{p,2}\cdot b_{p+1,1}=0$, for all $p$. In two landmark papers, Voisin has shown that the resolution of a general canonical curve of each genus is natural, see \cite{voisin-even}, \cite{voisin-odd}.

\vskip 3pt

We now fix a Prym curve of genus $g$, that is, a pair $[C, \eta]$, where $C$ is a smooth curve of genus $g$ and $\eta\neq \OO_C$ is a $2$-torsion point, $\eta^{\otimes 2}=\OO_C$. Prym curves of genus $g$ form an irreducible moduli space $\cR_g$ whose birational geometry is discussed  in \cite{FL} and \cite{farkas-verra-moduli-theta}. It has been conjectured in \cite{chiodo-eisenbud-farkas-schreyer}  that the resolution of a general Prym-canonical curve $$\phi_{K_C\otimes \eta}:C\rightarrow \PP^{g-2}$$ is natural. As explained in \cite{chiodo-eisenbud-farkas-schreyer}, this statement, which came to be known as the \emph{Prym-Green Conjecture}, reduces to one single vanishing statement in even genus, namely $$K_{\frac{g}{2}-3,2}\bigl(C, K_C\otimes \eta\bigr)=0,$$ and to the following two vanishing statements in odd genus:
\begin{equation}\label{pgparatlan}
K_{\frac{g-3}{2},1}\bigl(C, K_C\otimes \eta\bigr)=0 \ \ \mbox{ and } \ \ K_{\frac{g-7}{2},2}\bigl(C, K_C\otimes \eta\bigr)=0.
\end{equation}
The even genus case of the Prym-Green Conjecture is divisorial in moduli, that is, the locus
$\cZ_g:=\Bigl\{[C, \eta]\in \cR_g: K_{\frac{g}{2}-3,2}\bigl(C,K_C\otimes \eta\bigr)\neq 0\Bigr\}$ is the degeneracy locus of a morphism between vector bundles of the same rank over $\cR_g$.
The conjecture has been verified computationally in \cite{chiodo-eisenbud-farkas-schreyer} for all $g\leq 18$, with the exceptions of $g=8, 16$. The Prym-Green Conjecture is expected to fail for genera $g=2^n$, with $n\geq 3$, for reasons which are mysterious. We prove the following:

\begin{thm}\label{pgodd}
The Prym-Green Conjecture holds for a general Prym curve of odd genus.
\end{thm}

In particular, via Theorem \ref{pgodd}, we completely determine the shape of the resolution of a general Prym-canonical curve $C\subset \PP^{g-2}$ of odd genus $g=2i+5$. Precisely, if $S:=\mathbb C[x_0, \ldots, x_{g-2}]$, then the Prym-canonical ideal $I_{C}\subset S$ has the following resolution:

$$0\longleftarrow I_C\longleftarrow S(-2)^{b_{1,1}}\longleftarrow \cdots \longleftarrow S(-i)^{b_{i-1,1}}\longleftarrow S(-i-1)^{b_{i,1}}\oplus S(-i-2)^{b_{i,2}}\longleftarrow $$
$$\longleftarrow S(-i-3)^{b_{i+1,2}}\longleftarrow \cdots \longleftarrow S(-2i-4)^{b_{2i+2,2}}\longleftarrow 0,$$
where $$b_{p,1}=\frac{p(2i-2p+1)}{2i+3}{2i+4\choose p+1} \ \  \mbox{ for } p\leq i,\  \mbox{ and } \ \  b_{p,2}=\frac{(p+1)(2p-2i+1)}{2i+3}{2i+4\choose p+2} \  \ \  \mbox{ for } p\geq i.$$

In particular, the resolution is natural but fails to be \emph{pure}, precisely in the middle. In this sense, the resolution of a general Prym-canonical curve of \emph{odd} genus has the same shape as that of a general canonical curve of \emph{even} genus \cite{voisin-even}.

\vskip 3pt

The proof of Theorem \ref{pgodd} is by specialization to \emph{Nikulin surfaces}. A Nikulin surface  \cite{farkas-verra-moduli-theta}, \cite{vGS} is a $K3$ surface $X$ equipped with a double cover $f:\widetilde{X}\rightarrow X$ branched along eight disjoint rational curves $N_1, \ldots, N_8$. In particular, the sum of the exceptional curves is even, that is, there exists a class $\mathfrak{e}\in \mbox{Pic}(X)$ such that $\mathfrak{e}^{\otimes 2}=\OO_X(N_1+\cdots+N_8)$. If $C\subset X$ is a smooth curve of genus $g$ disjoint from the curves $N_1, \ldots, N_8$, then the restriction $\mathfrak{e}_C\in \mbox{Pic}^0(C)$ is a non-trivial point of order two, that is, $[C, \mathfrak{e}_C]\in \cR_g$.

\begin{thm}\label{hypnik}
Let $X$ be a general Nikulin surface endowed with a curve $C\subset X$ of odd genus $g\geq 11$, such that $C\cdot N_j=0$, for $j=1, \ldots, 8$. Then
$$ K_{\frac{g-3}{2},1}\bigl(C,K_C \otimes \mathfrak{e}_C\bigr)=0 \ \mbox{ and }\  K_{\frac{g-7}{2},2}\bigl(C,K_C \otimes \mathfrak{e}_C\bigr)=0.$$
In particular, the Prym-Green Conjecture holds generically on $\cR_g$.
\end{thm}

\vskip 7pt

The techniques developed for the Prym-Green Conjecture, in particular the use of $K3$ surfaces, we then use to study more generally syzygies of non-special line bundles on curves. In their influential paper \cite{green-lazarsfeld-projective}, Green and Lazarsfeld,  building on Green's Conjecture \cite{green-koszul} on syzygies of canonical curves, proposed the \emph{Secant Conjecture}, predicting that the shape of the resolution of a line bundle of sufficiently high degree is determined by its higher order ampleness properties: Precisely, if $L$ is a globally generated degree $d$ line bundle on a  curve $C$ of genus $g$ such that
\begin{equation}\label{ineq1}
d\geq 2g+p+1-2h^1(C,L) -\mbox{Cliff}(C),
\end{equation}
then $L$ fails property $(N_p)$ if and only if $L$ is not $(p+1)$-very ample, that is, the map $\phi_L:C\rightarrow \PP^{r}$ induced by the linear series $|L|$ embeds $C$ with a $(p+2)$-secant $p$-plane.  The case $h^1(C,L) \neq 0$ of the  Secant Conjecture easily reduces to the ordinary Green's Conjecture, that is, to the case $L=K_C$, see \cite{koh-stillman}.
If $L$ carries a $(p+2)$-secant $p$-plane, it is straightforward to see \cite{green-lazarsfeld-projective} that $K_{p,2}(C,L)\neq 0$, hence the Secant Conjecture concerns the converse implication. The Secant Conjecture is a refinement of Green's result \cite{green-koszul}, asserting that every line bundle $L\in \mbox{Pic}^{d}(C)$ with $d \geq 2g+p+1$  satisfies property $(N_p)$. Thus to study the Secant Conjecture for general curves, it suffices to consider the case of \emph{non-special} line bundles $L\in \mbox{Pic}^d(C)$, when $d\leq 2g+p$. Our first result answers completely this question in the case when both $C$ and $L$ are general:

\begin{thm}\label{secgen}
The Green-Lazarsfeld Secant Conjecture holds for a general curve $C$ of genus $g$ and a general line bundle $L$ of degree $d$ on $C$.
\end{thm}

For an integer $d\leq 2g+p$ and a non-special line bundle $L\in \mbox{Pic}^d(C)$, we consider the variety
$$V_{p+2}^{p+1}(L):=\bigl\{D\in C_{p+2}: h^0(C, L(-D))\geq d-g-p\bigr\}$$
of $(p+2)$-secants $p$-planes to $C$ in the embedding $\phi_L:C\rightarrow \PP^{d-g}$. The line bundles  possessing such a secant are those lying in a translate \emph{difference variety} in the Jacobian, precisely
\begin{equation}\label{diffv}
V_{p+2}^{p+1}(L)\neq \emptyset\ \Leftrightarrow \  L-K_C\in C_{p+2}-C_{2g-d+p}.
\end{equation}






For a general curve $C$, a general line bundle $L\in \mbox{Pic}^d(C) $ verifies $V_{p+2}^{p+1}(L) = \emptyset$ if and only if
\begin{equation}\label{ineq2}
d\geq g+2p+3.
\end{equation}
In particular, the Secant Conjecture (and Theorems \ref{secgen} and \ref{hatareset}) are vacuous  if inequality (\ref{ineq2}) is not fulfilled. Of particular importance is the \emph{divisorial case} of the Secant Conjecture, when the variety $\Bigl\{[C,L]\in \mathfrak{Pic}_g^d: K_{p,2}(C,L)\neq 0\Bigr\}$ is the degeneration locus of a morphism between vector bundles of the same rank over the universal degree $d$ universal Jacobian $\mathfrak{Pic}_g^d$. This divisorial case can be recognized by requiring that both inequalities (\ref{ineq1}) and (\ref{ineq2}) be equalities. We find
$$g=2i+1,  \  \ d=2g=4i+2, \  \ p=i-1.$$ In this case, we establish the Secant Conjecture in its strongest form, that is, for all $C$ and $L$:

\begin{thm}\label{diveset}
The Secant Conjecture  holds for every smooth curve $C$ of odd genus $g$ and every line bundle $L\in \mathrm{Pic}^{2g}(C)$, that is, one has the equivalence
$$K_{\frac{g-3}{2},2}(C,L)\neq 0  \ \Leftrightarrow \ \mathrm{Cliff}(C)<\frac{g-1}{2} \  \  \mathrm{or} \  \  \ L-K_C \in C_{\frac{g+1}{2}}-C_{\frac{g-3}{2}}.$$
\end{thm}

In the extremal case of the Secant Conjecture in even genus, that is, when $\mbox{deg}(L)=2g+1$, we prove the expected statement for \emph{general} curves and arbitrary line bundles on them:

\begin{thm}\label{even}
The Secant Conjecture holds for a Brill--Noether--Petri general curve $C$ of even genus and every line bundle $L\in \mathrm{Pic}^{2g+1}(C)$, that is,
$$K_{\frac{g}{2}-1,2}(C,L)\neq 0 \ \Leftrightarrow \ L-K_C\in C_{\frac{g}{2}+1}-C_{\frac{g}{2}-2}.$$
\end{thm}

Theorem \ref{even} is proved in Section 6, by regarding this case as a limit of the divisorial situation, occurring for odd genus. The strategy for proving Theorem \ref{diveset} is to interpret the G-L Secant Conjecture as an equality of divisors on the moduli space
$\cM_{g, 2g}$ of $2g$-pointed smooth curves of genus $g$ via the global Abel-Jacobi map $\cM_{g,2g}\rightarrow  \mathfrak{Pic}_g^{2g}$. For $g=2i+1$, denoting by $\pi:\cM_{g,2g}\rightarrow \cM_g$ the morphism forgetting the marked points, we consider the Hurwitz divisor $\cM_{g, i+1}^1:=\{[C]\in \cM_g: W^1_{i+1}(C)\neq \emptyset\}$ of curves with non-maximal Clifford index and its pull-back $\mathfrak{Hur}:=\pi^*(\cM_{g, i+1}^1)$.  We introduce two further divisors on $\cM_{g,2g}$. Firstly, the locus described by the condition that the embedding
$\phi_{\OO_C(x_1+\cdots+x_{2g})}:C\hookrightarrow \PP^{g}$ have extra syzygies
$$\mathfrak{Syz}:=\Bigl\{[C,x_1, \ldots, x_{2g}]\in \cM_{g,2g}: K_{i-1,2}\bigl(C, \OO_C(x_1+\cdots+x_{2g})\bigr)\neq 0\Bigr \}.$$
Theorem \ref{secgen} implies that $K_{i-1,2}(C, L)=0$ for a general line bundle $L$ of degree $2g$ on a general curve $C$, that is, $\mathfrak{Syz}$ is indeed a divisor on the moduli space.
Secondly, we consider the locus of pointed curves such that the image of $\phi_{\OO_C(x_1+\cdots+x_{2g})}$ has an $(i+1)$-secant $(i-1)$-plane, that is,
$$\mathfrak{Sec}:=\Bigl\{[C,x_1, \ldots, x_{2g}]\in \cM_{g,2g}: \OO_C(x_1+\cdots+x_{2g})\in K_C+C_{i+1}-C_{i-1}\Bigr \}.$$
Obviously $\mathfrak{Sec}$ is a divisor on $\cM_{g,2g}$ and using \cite{green-lazarsfeld-projective}, we know that the difference $\mathfrak{Syz}-\mathfrak{Sec}$ is effective. For any curve $C$ with maximal Clifford index, the cycles $K_C+C_{i+1}-C_{i-1}$ and $\bigl\{L\in \mbox{Pic}^{2g}(C):K_{i-1,2}(C,L)\neq 0\bigr\}$ have the same class in $H^*(\mbox{Pic}^{2g}(C), \mathbb Q)$. It follows that $$\mathfrak{Syz}=\mathfrak{Sec}+\pi^*(\mathfrak{D}),$$ where $\mathfrak{D}$ is an effective divisor on $\cM_g$. Via calculations in the Picard group of the moduli space $\mm_{g,2g}$, we shall establish the equality of divisors $\mathfrak{D}=\cM_{g,i+1}^1$, and thus prove the G-L Secant Conjecture in the case $d=2g$. An essential role is played by the following calculation:

\begin{thm}\label{osztaly}
The class of the divisor $\mathfrak{Syz}$ on $\cM_{g,2g}$ is given by the formula:
$$[\mathfrak{Syz}]=\frac{1}{2i}{2i\choose i-1}\Bigl(-(6i+2)\lambda+(3i+1)\sum_{j=1}^{2g} \psi_j\Bigr)\in CH^1(\cM_{g,2g}).$$
\end{thm}
Here $\lambda$ is the Hodge class on $\cM_{g,2g}$, whereas $\psi_1, \ldots, \psi_{2g}$ are the cotangent
classes corresponding to marked points.  The divisor class $[\mathfrak{Sec}]$ has been computed in \cite{Fa} Theorem 4.2:
$$[\mathfrak{Sec}]=\frac{1}{2i}{2i\choose i-1}\Bigl(-\frac{18i^2+10i-2}{2i-1}\lambda+(3i+1)\sum_{j=1}^{2g} \psi_j\Bigr)\in CH^1(\cM_{g,2g}).$$ Comparing these expressions  to the class of the Hurwitz divisor, famously computed in \cite{HM},
$$[\overline{\mathfrak{Hur}}]=\frac{1}{2i}{2i\choose i-1}\Bigl(\frac{6(i+2)}{2i-1}\lambda-\frac{i+1}{2i-1}\delta_{\mathrm{irr}}-\cdots\bigr)\in CH^1(\mm_g),$$ we obtain the following equality of divisors at the level of $\cM_{g,2g}$:
\begin{equation}\label{picid}
[\mathfrak{Syz}]=[\mathfrak{Sec}]+i\cdot [\mathfrak{Hur}] \in CH^1(\cM_{g,2g}).
\end{equation}
For a curve $C$ with $\mbox{Cliff}(C)<i$, it is easy to show that $\mbox{dim } K_{i-1,2}(C,L)\geq i$ (see Proposition \ref{scroll}), which implies that the divisor $\pi^*(\mathfrak{D})-i\cdot \mathfrak{Hur}$ is still effective on $\cM_{g,2g}$. On the other hand, $\pi^*([\mathfrak{D}-i\cdot \cM_{g,i+1}^1])=0$; since the map $\pi^*:\mbox{Pic}(\cM_g)\rightarrow \mbox{Pic}(\cM_{g,2g})$ is known to be injective, and on $\cM_g$ there can exist no non-trivial effective divisor whose class is zero (see for instance \cite[Lemma 1.1]{Fa2}), we conclude that $\mathfrak{D}=i\cdot \cM_{g,i+1}^1$.

\vskip 5pt

When (\ref{ineq1}) becomes an equality and $C$ is general, we give an answer which is more precise than the one provided in Theorem \ref{secgen}, as to which line bundles have unexpected syzygies.

\begin{thm}\label{hatareset}
Let $C$ be a  curve of genus $g$ which is Brill--Noether--Petri general. For $p\geq 0$, we set $$d=2g+p+1-\mathrm{Cliff}(C).$$ If
$L\in \mathrm{Pic}^d(C)$ is a non-special line bundle such that the secant variety $V_{g-p-3}^{g-p-4}(2K_C-L)$ has the expected dimension $d-g-2p-4$, then
$K_{p,2}(C,L)=0$.
\end{thm}
The condition that $\mbox{dim } V_{g-p-3}^{g-p-4}(2K_C-L)$ be larger than expected can be translated into a condition on translates of divisorial difference varieties, and Theorem \ref{hatareset} can be restated:
\begin{equation}\label{transl}
\mbox{If } \ K_{p,2}(C,L)\neq 0, \ \mbox{ then } \  L-K_C+C_{d-g-2p-3}\subset C_{d-g-p-1}-C_{2g-d+p}.
\end{equation}

\vskip 2pt

Clearly, if $L\in \mbox{Pic}^d(C)$ fails to be $(p+1)$-very ample, that is, $L-K_C\in C_{p+2}-C_{2g-d+p}$, then $L$ also satisfies condition (\ref{transl}). We show in Section 2, that for a general curve $C$, the locus $$V(C):=\Bigl\{L\in \mbox{Pic}^{d}(C): \mbox{dim } V_{g-p-3}^{g-p-4}(2K_C-L)>d-g-2p-4\Bigr\}$$ is a subvariety of $\mbox{Pic}^d(C)$ of dimension $2g+2p-d+2$.

\vskip 3pt

We describe our proof of Theorem \ref{secgen}, starting with the case of odd genus $g=2i+1$. Comparing the inequalities (\ref{ineq1}) and (\ref{ineq2}), we distinguish two cases. If $p\geq i-1$, Theorem \ref{secgen} can be easily reduced to the case $d=2p+2i+4$. Observe that in this case, the secant locus $K_C+C_{p+2}-C_{2g-d+p}$ is a divisor inside
$\mbox{Pic}^d(C)$. We use $K3$ surfaces with Picard number two:

\begin{thm}\label{k3i}
We fix positive integers $p$ and $i \geq 1$ with $p\geq i-1$, $p \geq 1$ and a general $K3$ surface $X$ with $\mathrm{Pic}(X)=\mathbb Z\cdot C\oplus \mathbb Z\cdot H$, where
$$C^2=4i+2, \ H^2=4p+4, \ C\cdot H=2p+2i+4.$$
Then $K_{j,2}(C, H_{C})=0$ for $j \leq p$. In particular, the G-L Secant Conjecture holds for general line bundles of degree $2p+2i+4$ on $C$.
\end{thm}

The proof of Theorem \ref{k3i} relies on Green's \cite{green-koszul}  exact sequence
$$\cdots \rightarrow K_{p,2}(X,H)\rightarrow K_{p,2}(C,H_C)\rightarrow K_{p-1,3}\bigl(X,\OO_X(-C), H\bigr)\rightarrow \cdots,$$
where the Koszul cohomology group whose vanishing has to be established is the one in the middle.
We look at the cohomology groups on the extremes. For a smooth curve $D\in |H|$, via the Lefschetz hyperplane principle \cite{green-koszul}, we write the isomorphism $K_{p,2}(X,H)\cong
 K_{p,2}(D,K_D)$.
 Recall that Green's Conjecture \cite{green-koszul} predicts the equivalence
$$K_{p,2}(D,K_D)=0\Leftrightarrow p<\mbox{Cliff}(C).$$ Green's Conjecture  holds for curves on arbitrary $K3$ surfaces \cite{voisin-even} \cite{voisin-odd}, \cite{aprodu-farkas-green}. We show by a Brill-Noether argument in the style of \cite{lazarsfeld-bnp} that $\mbox{Cliff}(D)=p+1$, hence $K_{p,2}(D, K_D)=0$. Furthermore, we establish the isomorphism $$K_{p-1,3}\bigl(X,\OO_X(-C),H)\cong K_{p-1,3}(D, \OO_D(-C), K_D)$$ (see Proposition \ref{prop1}). This latter group is zero if and only if
\begin{equation}\label{diff2}
H^1\Bigl(D, \bigwedge^p M_{K_D}\otimes K_D^{\otimes 2}(-C_D)\Bigr)=0,
\end{equation}
where $M_{K_D}$ is the kernel bundle of the evaluation map $H^0(K_D)\otimes \OO_D\rightarrow K_D$. We produce a particular $K3$ surface with Picard number \emph{three}, on which both curves $C$ and $D$ specialize to hyperelliptic curves such that the condition (\ref{diff2}) is satisfied.

The other possibility, namely when $p\leq i-1$, can be reduced to the case when $p=i-1$, that is, $d=2g$. As already pointed out, this is the only divisorial case of  the G-L Secant Conjecture.

When the genus $g$ is even, we show in Section 4 the following result:
\begin{thm}\label{paros2}
Let $g=2i \geq 4$ be even. Let $C$ be a general curve of genus $2i$ and let $L$ be a general line bundle on $C$  of degree $2p+2i+3$, where $p+1\geq i$. Then $$K_{j,2}(C,L)=0 \text{ for $j \leq p$.}$$
\end{thm}
This is achieved via proving Theorem \ref{evenkoszul}. We specialize $C$ to a curve lying on $K3$ surface $X$ such that $\mbox{Pic}(X)=\mathbb Z\cdot C \oplus \mathbb Z\cdot H$, where  this time $C^2=4i-2$, $H^2=4i+4$ and $C\cdot H=2p+2i+3$. As in the odd genus case, the vanishing in question is established for the line bundle $L=H_{C}$. The lattice theory required to show that $K_{j,2}(X,H)=0$ for $j\leq p$ in the even genus case is more involved that for odd genus, but apart from this, the two cases proceed along similar lines.

\vskip 7pt
We close the Introduction by discussing the connection between the Prym-Green and the Secant Conjecture respectively. First, observe that in odd genus $g=2i+5$, Prym-canonical line bundles fall within the range in which inequality (\ref{ineq1}) holds. In particular, the vanishing $K_{i-1,2}(C,K_C\otimes \eta)=0$ is predicted by Theorem \ref{hatareset}. Overall however, the Prym-Green Conjecture lies \emph{beyond} the range covered by the Secant Conjecture. For instance, we have seen that $K_{i,2}(C,K_C\otimes \eta)\neq 0$, despite the fact that a general Prym-canonical line bundle $K_C\otimes \eta$ is $(i+1)$-very ample (equivalently, $\eta\notin C_{i+2}-C_{i+2}$, see \cite[Theorem C]{chiodo-eisenbud-farkas-schreyer}).

\vskip 3pt

\noindent {\bf Structure of the paper:} Section 2 contains generalities on syzygies on curves and $K3$ surfaces, as well as considerations on difference varieties that enable us to reduce the number of cases in the Secant Conjecture. Sections 3 and 4 are lattice-theoretic in nature and present the proofs via $K3$ surfaces of the G-L Secant Conjecture in its various degree of precision (Theorems \ref{secgen} and \ref{k3i} respectively). Section 5 is concerned with syzygies of Nikulin surfaces and the proof of the Prym-Green Conjecture. Finally, in Section 6, we carry out calculations on the space $\cM_{g,2g}$ needed to complete the proof of the divisorial case of the Secant Conjecture (Theorem \ref{diveset}).

\vskip 3pt

\noindent {\bf Acknowledgment:} We are grateful to M. Aprodu, D. Eisenbud, J. Harris, R. Lazarsfeld, F.-O. Schreyer, and especially to C. Voisin for many useful discussions related to this circle of ideas.

\section{Generalities on Syzygies of Curves}
In this section we gather some general results on syzygies of curves which will be of use. We fix a globally generated line bundle $L$ and a sheaf $\F$ on a projective variety $X$.
We then form the graded $S:=\mbox{Sym } H^0(X,L)$-module
$$\Gamma_X(\F,L):=\bigoplus_{q\in \mathbb Z} H^0(X, \F\otimes L^{\otimes q}).$$
Following \cite{green-koszul}, we denote by $K_{p,q}(X,\F,L)$ the space of $p$-th syzygies of weight $q$ of the module $\Gamma_X(\F,L)$. Often $\F=\OO_X$, in which case we denote
$K_{p,q}(X,L):=K_{p,q}(X,\OO_X,L)$. Geometrically, one studies Koszul cohomology groups via kernel bundles. Consider the vector bundle
$$M_{L} := \text{Ker}\bigl\{H^0(X,L) \otimes \mathcal{O}_X \twoheadrightarrow L\bigr\},$$ where the above map is evaluation. We quote the following description from
\cite{lazarsfeld-kernel}:

\begin{align*}
K_{p,q}(X, \mathcal{F}, L) & \simeq \text{coker} \Bigl\{\bigwedge^{p+1}H^0(X,L) \otimes H^0(X,\mathcal{F} \otimes L^{q-1}) \to H^0(X,\bigwedge^p M_L \otimes \mathcal{F} \otimes L^q)\Bigr\} \\
& \simeq \text{ker}\Bigl\{H^1(X,\bigwedge^{p+1} M_L \otimes \mathcal{F} \otimes L^{q-1}) \to \bigwedge^{p+1}H^0(X, L) \otimes H^1(X,\mathcal{F} \otimes L^{q-1})\Bigr\}
\end{align*}

In particular, for a non-special line bundle $L$, we have the equivalence, cf. \cite[Lemma 1.10]{GL3}:
$$K_{p,2}(X,L)=0\Leftrightarrow H^1\Bigl(X, \bigwedge^{p+1}M_L\otimes L\Bigr)=0.$$

Using the above description of Koszul cohomology, it follows that the difference of  Betti numbers on any diagonal of the Betti diagram of a non-special line bundle $L$ on a curve $C$ is an Euler characteristic of a vector bundle on $C$, hence constant. Precisely,
\begin{equation}\label{diffbetti}
\mbox{dim } K_{p+1,1}(C,L)-\mbox{dim } K_{p,2}(C,L)=(p+1)\cdot {d-g\choose p+1}\Bigl(\frac{d+1-g}{p+2}-\frac{d}{d-g}\Bigr).
\end{equation}

The following fact is well-known and essentially trivial:
\begin{prop} \label{reduction-by-one}
Let $C$ be a smooth curve and $L$ a globally generated line bundle on $C$ with $h^1(C,L)=0$ and $K_{p,2}(C,L)=0$. If $x \in C$ is a point such that $L(-x)$ is globally generated, then $K_{p-1,2}(C,L(-x))=0$.
\end{prop}
\begin{proof}
Follows via the above description of Koszul cohomology, by using the exact sequence
$$0 \to \bigwedge^p M_{L(-x)} \otimes L \to \bigwedge^p M_L \otimes L \to \bigwedge^{p-1}M_{L(-x)} \otimes L(-x) \to 0. $$
\end{proof}

\subsection{Syzygies of $K3$ surfaces.} The following result, while simple, is essential in the proof of Theorem \ref{secgen}, for it allows us to ultimately reduce the vanishing required in the G-L Secant Conjecture to a vanishing of the type appearing in a slightly different context in the statement of the Minimal Resolution Conjecture of \cite{farkas-popa-mustata}.
\begin{lem}
Let $X$ be a K3 surface, and let $L$ and $H$ be line bundles with $H$ effective and base point free. Assume $(H \cdot L)>0$ and $H^1(X,qH-L)=0$ for $q \geq 0$. Then for each smooth curve $D \in |H|$, we have that $K_{p,q}(X,-L,H) \simeq K_{p,q}(D,-L_D,K_D)$ for all $p$ and $q$.
\end{lem}
\begin{proof}
The proof proceeds along the lines of \cite[Theorem  3.b.7]{green-koszul}. By the assumptions, we have a short exact sequence of $\mbox{Sym } H^0(X,H)$-modules
$$0 \to \bigoplus_{q \in \mathbb{Z}} H^0\bigl(X,(q-1)H-L\bigr) \to \bigoplus_{q \in \mathbb{Z}} H^0\bigl(X,qH-L\bigr) \to  \bigoplus_{q \in \mathbb{Z}} H^0\bigl(D,qK_D-L_D\bigr) \to 0,$$ where the first map is multiplication by the section $s \in H^0(X,H)$ defining $D$. Let $B$ denoted the graded $\mbox{Sym } H^0(X,H)$-module $\bigoplus_{q \in \mathbb{Z}} H^0(D,qK_D-L_D)$, and write its associated Koszul cohomology groups as $K_{p,q}(B, H^0(X,H))$. The above short exact sequence induces a long exact sequence at the level of Koszul cohomology
\begin{equation}\label{koszulseq}
\cdots \to K_{p,q-1}(X,-L,H) \to K_{p,q}(X,-L,H) \to K_{p,q}(B, H^0(X,H)) \to K_{p-1,q}(X,-L,H) \to \cdots
\end{equation}
The maps $K_{p,q-1}(X,-L,H) \to K_{p,q}(X,-L,H)$ are induced by multiplication by $s$, and hence are zero, see also \cite[1.6.11]{green-koszul}.
Choose a splitting
$$ H^0(X,H) \simeq \C \{s \} \oplus H^0(D,K_D).$$
This induces isomorphism
$$f_p \; : \; \bigwedge^p H^0(X,H) \simeq \bigwedge^{p-1} H^0(D,K_D) \oplus \bigwedge^{p} H^0(D,K_D).$$

The Koszul cohomology of the module $B$ is computed by the cohomology of the complex $\mathcal{A}_{\bullet}$
 \begin{align*}
 \to &\bigwedge^p H^0(D,K_D)\otimes H^0\bigl(D,(q-1)K_D-L_D\bigr) \bigoplus \bigwedge^{p+1} H^0(D,K_D)\otimes H^0\bigl(D,(q-1)K_D-L_D\bigr) \to \\
  &\bigwedge^{p-1} H^0(D,K_D)\otimes H^0(D,qK_D-L_D) \bigoplus \bigwedge^{p} H^0(D,K_D)\otimes H^0(D,qK_D-L_D) \to \cdots
 \end{align*}
 where the maps being equal to $(-d_{p,q-1}, d_{p+1,q-1})$, with
 $$d_{p+1,q-1}: \bigwedge^{p+1}H^0(D,K_D)\otimes H^0\bigl(D,(q-1)K_D-L_D\bigr) \to \bigwedge^{p}H^0(D,K_D)\otimes H^0(D,qK_D-L_D)$$ and its shift $d_{p, q-1}$ being Koszul differentials
 (note that we have used that $s$ vanishes along $D$).
 Thus we have
 $$ K_{p,q}(B, H^0(X,H)) \simeq K_{p,q}(D,-L_D,K_D) \oplus K_{p-1,q}(D,-L_D,K_D)$$ and hence via the sequence (\ref{koszulseq}), we obtain
 $$ K_{p,q}(X,-L,H) \oplus K_{p-1,q}(X,-L,H) \simeq K_{p,q}(D,-L_D,K_D) \oplus K_{p-1,q}(D,-L_D,K_D).$$
 For $p=0$, this becomes $K_{0,q}(X,-L,H) \simeq K_{0,q}(D,-L_D,K_D)$. The claim follows by induction on $p$.
\end{proof}

Now let, $X$ be a K3 surface, and let $L,H\in \mbox{Pic}(X)$ and $D\in |H|$ be as in the hypotheses of the lemma above. Assume further that $H^0(X,H-L)=0$, and let $C \in |L|$ be a smooth, integral curve. Following \cite[Theorem  3.b.1]{green-koszul}, we have a long exact sequence
$$ \cdots \to K_{p,q}(X,H) \to K_{p,q}(C,H_C) \to K_{p-1,q+1}(X,-L,H) \simeq K_{p-1,q+1}(D,-L_D,K_D)\to \cdots $$
Thus we have:
\begin{prop} \label{prop1}
In the above situation, assume  $K_{p,q}(X,H)=K_{p-1,q+1}(D,-L_D,K_D)=0$. Then $K_{p,q}(C,H_C)=0$.
\end{prop}

We will also make use of the following  result of Mayer's \cite[Proposition 8]{mayer-families}.
\begin{prop} \label{bpf-prop}
Let $X$ be a K3 surface and let $L \in \mathrm{Pic}(X)$ be a big and nef line bundle. Assume there is no smooth elliptic curve
$F$ with $(F \cdot L)=1$. Then $L$ is base point free.
\end{prop}

We now turn our attention to the Green-Lazarsfeld Secant Conjecture \cite{green-lazarsfeld-projective}. We fix a general curve $C$ of genus $g$ and an integer $d\geq 2g+1-\mbox{Cliff}(C)$. Using
\cite[p. 222]{ACGH}, we observe that a general line bundle $L\in \mbox{Pic}^d(C)$ is projectively normal, that is, the multiplication maps
$$\mbox{Sym}^n H^0(C,L)\rightarrow H^0(C, L^{\otimes n})$$ are surjective for all $n$. Since obviously also $H^1(C,L)=0$, via \cite[Lemma 1.10]{GL3}, we conclude that $\phi_L:C\rightarrow \PP^{d-g}$
verifies property $(N_p)$ for some integer $p\geq 0$ with $d\geq 2g+p+1-\mbox{Cliff}(C)$, (that is, $K_{j,2}(C,L)=0$ for all $j\leq p$), if  and only if one \emph{single} vanishing $K_{p,2}(C,L)=0$ holds.

\vskip 4pt

 Proposition \ref{reduction-by-one} is used to reduce the proof of Theorem \ref{secgen} to the following cases:

\begin{prop}\label{reduction}
 Let $C$ be a general curve of genus $g$. In order to conclude that the G-L Secant Conjecture holds for $C$ and for general line bundles on $C$ in each degree, it suffices to exhibit a non-special line bundle $L\in \mathrm{Pic}^d(C)$  such that $K_{p,2}(C,L)=0$, in each of the following cases:
\begin{enumerate}
\item $g=2i+1$, $d=2p+2i+4$ \ and \ $p \geq i-1$
\item $g=2i$, $d=2p+2i+3$ \ and \ $p \geq i-1$.
\end{enumerate}
\end{prop}
\begin{proof}
As $C$ is general, $\text{Cliff}(C) = \lfloor \frac{g-1}{2} \rfloor$. We explain the case $g=2i+1$, the remaining even genus case being similar.
In the case $p \geq i-1$, the two inequalities (\ref{ineq1}) and (\ref{ineq2}), that is,
$$d \geq 2g+p+1-\text{Cliff}(C)  \ \mbox{ and } \ d \geq g+2p+3$$ respectively, reduce to the single inequality
$$d \geq 2p+2i+4. $$
If $d$ is even, we write $d=2q+2i+4$, where $q\geq p$. By assumption, we can find a line bundle $L\in \mbox{Pic}^d(C)$ with $K_{q,2}(C,L)=0$.  We may assume $L$ to be projectively normal and then, as explained, it follows that $K_{p,2}(C,L)=0$. If $d$ is odd, we write $d=2q+2i+5$, where again $q\geq p$. By assumption, there exists a line bundle $L'\in \mbox{Pic}^{d+1}(C)$ with $K_{q+1}(C,L')=0$. We set $L:=L'(-x)$, where $x\in C$ is a general point. By Proposition \ref{reduction-by-one}, we find that $K_{q,2}(C,L)=0$, hence $K_{p,2}(C,L)=0$ as well.

\vskip 3pt

In the range $p \leq i-1$, the inequalities (\ref{ineq1}) and (\ref{ineq2}) reduce to the inequality
$$d\geq 3i+p+3 \ \bigl(=2g+p+1-\mbox{Cliff}(C)\bigr).$$ If $d\leq 4i+2$, we apply the assumption in degree $4i+2$ to find a line bundle $L'\in \mbox{Pic}^{4i+2}(C)$ such that $K_{i-1,2}(C,L')=0$. We then choose a general effective divisor $D\in C_{4i+2-d}$ and set $L:=L'(-D)\in \mbox{Pic}^d(C)$. Via Proposition \ref{reduction-by-one}, we conclude that $K_{d-3i-3,2}(C,L)=0$, hence $K_{p,2}(C,L)=0$, as well. If, on the other hand $d\geq 4i+2$, then for even degree, we write $d=2q+2i+4$, where $q\geq i-1$ and then apply the assumption as in the previous case. The case when $d$ is odd is analogous.
\end{proof}

\vskip 3pt

Especially significant in our study is the divisorial case of the G-L Secant Conjecture:
$$g=2i+1, \ d=4i+2, \ p=i-1.$$
Using (\ref{diffbetti}), note that in this case $\mbox{dim } K_{i-1,2}(C,L)=\mbox{dim } K_{i,1}(C,L)$.
For a pair $[C,L]\in \mathfrak{Pic}^{2g}_g$, we consider the embedding $\phi_L:C\hookrightarrow \PP^{g}$ and denote by $I(L)$ the graded ideal of $\phi_L(C)$. The next observation, to be used in the proof of Theorem \ref{hatareset}
provides a lower bound for the number of syzygies of $L$, when the curve $C$ has Clifford index less than maximal.

\begin{prop}\label{scroll}
Let $C$ be a smooth curve of genus $g=2i+1$ having gonality at most $i+1$ and $L$ a line bundle of degree $2g$ on $C$. Then
$\mathrm{dim}  \ K_{i,1}(C,L)\geq i.$
\end{prop}
\begin{proof} It suffices to prove the statement for a general $(i+1)$-gonal curve $[C]\in \cM_{2i+1,i+1}^1$. We fix such a $C$, then denote by $A\in W^1_{i+1}(C)$ a pencil of minimal degree and by $\phi_A:C\rightarrow \PP^1$ the associated covering. Let $L\in \mbox{Pic}^{2g}(C)$ a suitably general line bundle such that $\phi_L(C)$ is projectively normal, $h^0(C, L\otimes A^{\vee})=i+1$ and the multiplication map
$$\mu_{A,L\otimes A^{\vee}}: H^0(C,A)\otimes H^0(C, L\otimes A^{\vee})\rightarrow H^0(C,L)$$
is an isomorphism. The vector bundle $E:=(\phi_A)_{*}(L)$ induces an $(i+1)$-dimensional  scroll $X:=\PP(E)\hookrightarrow \PP H^0(C,L)^{\vee}$ of degree $i+1$ containing $\phi_L(C)$. The equations of $X$ inside $\PP^{2i+1}$ are obtained by taking the $(2\times 2)$-minors of the matrix describing the map $\mu_{A, L\otimes A^{\vee}}$, see \cite{Sch}. Precisely, if $\{\sigma_1, \sigma_2\}$ is a basis of $H^0(C,A)$ and $\{\tau_1, \ldots, \tau_{i+1}\}$ is a basis of $H^0(C,L\otimes A^{\vee})$ respectively, then $X$ is cut out by the quadrics
$$q_{\ell}:= (\sigma_1 \tau_{\ell})\cdot (\sigma_2  \tau_{i+1})-(\sigma_1 \tau_{i+1})\cdot (\sigma_2 \tau_{\ell}) \in K_{1,1}(C,L),$$
where $\ell=1, \ldots, i$. Recall that there is an isomorphism
$$K_{i,1}(C,L)\cong K_{i-1,2}\bigl(I(L), H^0(C,L)\bigr):=\mbox{Ker}\Bigl\{\bigwedge^{i-1} H^0(C,L)\otimes I_2(L)\rightarrow \bigwedge^{i-2} H^0(C,L)\otimes I_3(L)\Bigr\}.$$
By direct computation, for $\ell=1, \ldots, i$, we write down the following syzygies:
$$\gamma_{\ell}:=\sum_{j=1}^i(\sigma_2\tau_1)\wedge \ldots \wedge (\widehat{\sigma_2 \tau_{j}})\wedge \ldots \wedge (\sigma_2\tau_i)\otimes q_{\ell}\in \bigwedge ^{i-1} H^0(C,L)\otimes I_2(L).$$
Since the quadrics $q_1, \ldots, q_i$ are independent in $K_{1,1,}(C,L)$, we find that $\mbox{dim } K_{i,1}(C,L)\geq i.$
\end{proof}
\subsection{ Syzygies and translates of difference varieties} For a curve $C$ and $a,b\geq 0$, let
$$C_a-C_b\subset \mbox{Pic}^{a-b}(C)$$
be the difference variety, consisting of line bundles of the form $\OO_C(D_{a}-E_b)$, where $D_{a}$ and $E_b$ are effective divisors on $C$ having degrees $a$ and $b$ respectively. A result of \cite{farkas-popa-mustata} provides an identification of the divisorial difference variety, valid for each smooth curve of genus $g$:
\begin{equation}\label{raynaud}
C_{g-j-1}-C_j=\Theta_{\bigwedge ^j Q_C}.
\end{equation}
The right-hand-side here denotes the \emph{theta divisor} of the vector bundle $Q_{C}:=M_{K_C}^{\vee}$, that is,
$$\Theta_{\bigwedge ^j Q_C}:=\Bigl\{\xi \in \mbox{Pic}^{g-2j-1}(C): h^0\bigl(C, \bigwedge^j Q_C\otimes \xi)\geq 1\Bigr\}.$$

We fix non-negative integers $j\leq g-1$ and $a\leq g-j-1$ and introduce the cycle
$$V(C):=\Bigl\{L\in \mbox{Pic}^{g-2j-a-1}(C): L+C_a\subset C_{g-j-1}-C_j\Bigr\}.$$ Obviously $C_{g-j-a-1}-C_j\subset V(C)$. For hyperelliptic curves,  one has set-theoretic equality:

\begin{prop}\label{diffvar1}
Let $C$ be a smooth hyperelliptic curve of genus $g$ and fix integers $j\leq g-1$ and $a\leq g-j-1$.  The following equivalence holds for a line bundle
$L  \in \mathrm{Pic}^{g-2j-a-1}(C)$:
$$L+C_a\subset C_{g-j-1}-C_j \ \Leftrightarrow \  L\in C_{g-j-a-1}-C_j.$$
\end{prop}

\begin{proof}
Via (\ref{raynaud}), the hypothesis $L+C_a\subset C_{g-j-1}-C_j$ can be reformulated cohomologically:
$$h^0\Bigl(C, \bigwedge ^j M_{K_C}\otimes K_C(-L-D_a)\Bigr)\geq 1, \mbox{ for every effective divisor } D_a\in C_a.$$
If $C$ is hyperelliptic and $A\in W^1_2(C)$ denotes the hyperelliptic pencil, then the kernel bundle splits $M_{K_C}=\bigl(A^{\vee})^{\oplus (g-1)}$, and the previous condition translates into $h^0\bigl(C, A^{\otimes {(g-j-1)}}\otimes (-L-D_a)\bigr)\geq 1$. Since a linear series on a curve can have only finitely many base points, we obtain that
$$h^0(C, (g-1-j)A- L)\geq a+1.$$
On $C$ any complete linear series $\mathfrak g^r_d$ with $d\leq 2g-2$ is of the form
$A^{\otimes r}(y_1+\cdots +y_{d-2r})$,  with  $y_1, \ldots, y_{d-2r}\in C$. It follows that there exists a divisor $E=x_1+\cdots +x_{g-a-1}$ such that $L=A^{\otimes (g-j-a-1)}(-E)$. Denoting by $x_{\ell}'\in C$ the hyperelliptic conjugate of $x_{\ell}$, we  obtain
$$L=\OO_C\bigl(x_1'+\cdots+x_{g-1-a-j}'-x_{g-a-j}-\cdots -x_{g-a-1}\bigr)\in C_{g-j-a-1}-C_j.$$
\end{proof}

A consequence of the above is that for a general curve $C$, we have an equality of cycles
$$V(C)=(C_{g-j-a-1}-C_j)+V'(C),$$ where $V'(C)$ is a residual cycle of dimension at most $g-a-1$. Contrary to our initial expectation, which was later tempered by Claire Voisin, the residual cycle $V'(C)$ can in general be non-empty and have dimension much smaller than $g-a-1$, as the following example shows:

\begin{prop}\label{diffcon} Let $C$ be a general curve of genus $g=2i$ and $A\in W^1_{i+1}(C)$ a pencil of minimal degree. We set $L:=K_C-2A\in \mathrm{Pic}^{2i-4}(C)$. Then $L+C\subset C_{2i-2}-C$, but however $L\notin C_{2i-3}-C$.
\end{prop}
\begin{proof}
For a point $x\in C$, let $p_1+\cdots+p_i\in C_i$ be the divisor such that $x+\sum_{j=1}^i p_j\in |A|$. Since $h^0(C, K_C\otimes A^{\vee})=i$, there exists an effective divisor $D\in C_{2i-2}$ with $D+\sum_{j=1}^{i-1} p_j\in |K_C\otimes A^{\vee}|$. It follows, that
$$L(x)=\bigl(K_C\otimes A^{\vee}\bigr)\otimes A^{\vee}(x)=\OO_C(D+p_1+\cdots+p_{i-1})\otimes A^{\vee}(x)=\OO_C(D-p_i)\in C_{2i-2}-C.$$
On the other, we claim that $L\notin C_{2i-3}-C$. Else, there exists a point $y\in C$ such that $H^0(C, L(y))\neq 0$. Via the Base Point Free Pencil Trick, this is equivalent to saying that the multiplication map $H^0(C,A)\otimes H^0(C, K_C\otimes A^{\vee}(y))\rightarrow H^0(C,K_C(y))$ is not injective. Since $h^0(C, K_C\otimes A^{\vee}(y))=h^0(C, K_C\otimes A^{\vee})$, this implies that the Petri map associated to the pencil $A$ is not injective, a contradiction.
\end{proof}
\begin{remark} The structure of the residual cycle $V'(C)$ remains mysterious. One case that is understood via the exercises in \cite[p.276 ]{ACGH} is that when $C$ is a general curve of genus $4$ and $a=j=1$. Then $V(C)=(C-C)+\bigl\{\mp(K_C-2A)\bigr\}$, where $A\in W^1_3(C)$, that is, $V'(C)$ is a $0$-dimensional cycle.
\end{remark}

\section{The generic Green--Lazarsfeld secant conjecture for curves of odd genus}

In this section we prove Theorem \ref{k3i}. We begin by recalling a few basic facts. Let $\mathfrak{h}$ be an even lattice of rank $\rho+1 \leq 10$ and signature $(1, \rho)$. The moduli space of $\mathfrak{h}$-polarized K3 surfaces exists as a quasi-projective algebraic variety, is nonempty, and has at most two components both of dimension $19-\rho$, which locally on the period domain are interchanged by complex conjugation, \cite{dolgachev}. Complex conjugation here means that a complex surface $X$ with complex structure $J$ is sent to $(X,-J)$.

We fix integers $g=2i+1$ with $i \geq 1$ and $p\geq i-1$, $p \geq 1$. Let $\Theta_{g,p}$ be the rank two lattice with ordered basis $\{ H, \eta \}$ and intersection form:
\[ \left( \begin{array}{cc}
4p+4 & 2p-2i  \\
2p-2i& -4  \end{array} \right)\]
We let $L$ denote the class $H-\eta$. Notice that $(H \cdot L)=2p+2i+4$ and $(L)^2=4i$.
We denote by $\hat{\Theta}_{g,p}$ the rank three lattice with ordered basis $\{ H, \eta, E \}$ and intersection form
\[ \left( \begin{array}{ccc}
4p+4 & 2p-2i & 2 \\
2p-2i & -4  &0 \\
2 & 0 & 0
\end{array} \right).\] Obviously $\Theta_{g,p}$ can be primitively embedded in $\hat{\Theta}_{g,p}$. By the surjectivity of the period mapping, there exist smooth K3 surfaces $Z_g$ and\ $\hat{Z}_g$ respectively with Picard lattices isomorphic to $\Theta_{g,p}$ and $\hat{\Theta}_{g,p}$ respectively, and such that $H$ is big and nef.
\begin{lem}
Let $\alpha, \beta \in \hat{\Theta}_{g,p}$. Then $(\alpha \cdot \beta)$ is even and $(\alpha)^2$ is divisible by four.
\end{lem}
\begin{proof}
The first claim is clear, as all entries in the above rank three matrix are even. For the second claim, write $\alpha=aH+b\eta+cE$ for $a,b,c \in \mathbb{Z}$ and compute $$(\alpha)^2=4a^2(p+1)-4b^2+4ab(p-i)+4ac.$$
\end{proof}
 \begin{cor} \label{lattice-cor}
 Let  $Z_g$ respectively\ $\hat{Z}_g$ be K3 surfaces with Picard lattices isomorphic to $\Theta_{g,p}$ respectively\ $\hat{\Theta}_{g,p}$. Suppose a divisor $\alpha$ in $\Theta_{g,p}$ respectively\ $\hat{\Theta}_{g,p}$ is effective. Then $(\alpha)^2 \geq 0$ and $\alpha$ is base point free and nef.
 \end{cor}
 \begin{proof}
   Suppose a divisor $\alpha$ in $\Theta_{g,p}$ resp.\ $\hat{\Theta}_{g,p}$ is effective. Since there are no $(-2)$ classes in $\Theta_{g,p}$ or in $\hat{\Theta}_{g,p}$, necessarily $(\alpha)^2 \geq 0$ and $\alpha$ is nef. Since there do not exist classes $F$ with $(F)^2=0$, $(F \cdot \alpha)=1$, the class $\alpha$ is base point free by Proposition \ref{bpf-prop}.
 \end{proof}

Our next task is to study the Brill--Noether theory of curves in the linear system $|H|$.
\begin{lem}
There exists a K3 surface $Z_{g}$ with $\mathrm{Pic}(Z_{g}) \simeq  \Theta_{g,p}$ and $H$ big and nef. For a general such $K3$ surface, a general curve $D\in |H|$ is Brill--Noether--Petri general, in particular $\mathrm{Cliff}(D)=p+1$.
\end{lem}
\begin{proof}
We have an obvious primitive embedding $\Theta_{g,p} \hookrightarrow \hat{\Theta}_{g,p}$. Note that on $\hat{Z}_g$, any class of the form $aH+ b\eta\in \mbox{Pic}(\hat{Z}_g)$ with $a<0$ is not effective, as it has negative intersection with $E$. Let $Z_g$ be a general $\Theta_{g,p}$-polarized K3 surface which deforms to $\hat{Z}_g$; i.e.  a general element in an least one of the components of $\Theta_{g,p}$-polarized K3 surfaces. It suffices to establish that the hyperplane class admits no decomposition $H=A_1+A_2$ for divisors $A_1,A_2$ with $h^0(Z_g, A_i) \geq 2$, for $i=1,2$. Indeed, this follows from the proof of \cite[Lemma \ 1.3]{lazarsfeld-bnp}. For precise details, we refer  to \cite[Lemma \ 5.2]{kemeny-singular}).

Suppose we have such a decomposition on a general surface $Z_g$. Then the $A_i$ would deform to effective divisors on $\hat{Z}_g$, so we could write $A_i=a_iH+b_i \eta$ for $a_i, b_i \in \mathbb{Z}$, with $a_i \geq 0$ and $a_1+a_2=1$. Without loss of generality, we may assume that $a_1=0$, so $(A_1)^2=-4 b^2_1$. From Corollary \ref{lattice-cor}, this forces $b_1=0$, so $h^0(\hat{Z}_g, A_1)=1$, which is a contradiction. If $Z_g^c$ is the complex conjugate of $Z_g$ with the induced $\Theta_{g,p}$-polarization, then the claim above clearly also holds for the image of $H$ in $\mathrm{Pic}(Z_g^c)$.
\end{proof}
\begin{cor}
Let $Z_g$ be a general $\Theta_{g,p}$-polarized K3 surface. Then $K_{j,1}(Z_g,H)=0$ for $j \leq p$.
\end{cor}
\begin{proof}
From the above lemma, $\text{Cliff}(D)=p+1$ for $D \in |H|$ general. Thus the result follows from \cite[Thm.\ 1.3]{aprodu-farkas-green}.
\end{proof}

Whereas for a general $\Theta_{g,p}$-polarized $K3$ surface, general smooth curves $D\in |H|$ and $C\in |L|$ respectively, are Brill-Noether general, this is no longer the case for $\hat{\Theta}_{g,p}$-polarized $K3$ surfaces, when both $D$ and $C$ become hyperelliptic.
\begin{lem}\label{hypvanodd}
Let $\hat{Z}_g$ be a general $\hat{\Theta}_{g,p}$-polarized K3 surface. Then $$H^1\Bigl(D, \bigwedge^j M_{K_D}\otimes K_D\otimes \eta_D \Bigr)=0$$ for $j \leq p$ and for a general curve $D \in |H|$.
\end{lem}
\begin{proof}
As $D$ is hyperelliptic, we have the following splitting $$\bigwedge^{j} M_{K_D} \simeq \mathcal{O}_D(-jE)^{\oplus {2p+2\choose j}} $$ and $K_D\cong 2(p+1)E_D$, see \cite[Prop. 3.5]{farkas-popa-mustata}. We need to show $H^1\bigl(D,\eta_D((2p+2-j)E_D)\bigr)=0$. Since $((2p+2-j)E+\eta)^2=-4$, we have $h^0(\hat{Z}_g,(2p+2-j)E+\eta)=h^2(\hat{Z}_g,(2p+2-j)E+\eta)=0$ and thus $h^1(\hat{Z}_g,(2p+2-j)E+\eta)=0$, using Corollary \ref{lattice-cor}. Lastly, we compute
$$((2p+2-j)E+\eta-H)^2=4(i+j)-8p-8 \leq 4(i-p)-8 \leq -4$$ and therefore $h^2( \hat{Z}_g, (2p+2-j)E+\eta-H)=0$. Thus $H^1(D,\eta_D((2p+2-j)E_D)\bigr)=0$ as required.
\end{proof}
As an immediate corollary we have:
\begin{cor}
There is a nonempty open subset of the moduli space of $\Theta_{g,p}$-polarized K3 surfaces such that $K_{j-1,3}(D,\eta_D-K_D, K_D)=0$ for all $j\leq p$ and $D \in |H|$.
\end{cor}
We may now conclude this section by establishing the G-L Secant Conjecture for general line bundle on general curves of odd genus.

\vskip 3pt
\noindent \emph{Proof of Theorem \ref{k3i}.} We retain the notation. We have established that there exists a K3 surface $Z_g$ with Picard lattice $\Theta_{g,p}$ such that $H$ is big and nef, $K_{j,1}(Z_g,H)=0$ and $K_{j-1,3}(D,-L_D,K_D)=0$ for $j\leq p$ for each smooth curve $D \in |H|$. We have $h^0(Z_g, \eta)=h^1(Z_g, \eta)=0$, $h^1(Z_g, L)=0$. For $q \geq 2$ we compute
\begin{align*}
(qH-L)^2&=(q-1)^2(4p+4)+4(q-1)(p-i)-4 \\
& \geq (q-1)((q-1)(4p+4)-4)-4\ \geq 4p-4.
\end{align*}
Since $p \geq 1$, $H^1(Z_g, qH-L)=0$ for $q \geq 2$, using Corollary \ref{lattice-cor} (in the case $p=1$, use that $qH-L$ is primitive). Thus Proposition \ref{prop1} applies, and for each smooth curve $C\in |L|$ we have $K_{j,2}(C,H_C)=0$ for $j \leq p$. Note finally that $h^0(C,H_C)=h^0(Z_g,H)=2p+4$, that is, $H_C$ is non-special.
\hfill $\Box$

\section{The generic Green--Lazarsfeld Secant Conjecture for curves of even genus}
 Suppose $g=2i$ for $i \geq 2$ and $p\geq i-1$. Let $\Xi_{g,p}$ be the following rank two lattice with ordered basis $\{ H, \eta \}$ and intersection form
\[ \left( \begin{array}{cc}
4p+4 & 2p-2i+1  \\
2p-2i+1 & -4  \end{array} \right).\] We set $L=H-\eta$ and note that $(H \cdot L)=2p+2i+3$ and $(L)^2=2g-2$.
We let $\hat{\Xi}_{g,p}$ be the rank three lattice of signature $(1,2)$ having ordered basis $\{ H, \eta, E \}$ and intersection form
\[ \left( \begin{array}{ccc}
4p+4 & 2p-2i+1 & 2 \\
2p-2i+1 & -4  &0 \\
2 & 0 & 0
\end{array} \right).\] Obviously $\Xi_{g,p}$ can be primitively embedded in $\hat{\Xi}_{g,p}$. By the Torelli theorem \cite{dolgachev}, there exist smooth K3 surfaces $Z_g$ respectively $\hat{Z}_g$ with Picard lattices isomorphic to $\Xi_{g,p}$ respectively $\hat{\Xi}_{g,p}$, and such that $H$ is big and nef.

\begin{lem}
Let $\hat{Z}_g$ be a K3 surface with Picard lattice given by $\hat{\Xi}_{g,p}$ and such that $H$ is big and nef. Then $E$ and $H$ are effective and base point free.
\end{lem}
\begin{proof}
Both $E$ and $H$ are effective since they have positive intersection with $H$. We claim that $E$ is base point free. As $(E)^2=0$, it suffices to prove that $E$ is nef, \cite[Prop.\ 2.3.10]{huy-lec-k3}. Let $R=aH+bE+c \eta$ be the class of a smooth rational curve, for $a,b,c \in \mathbb{Z}$, and assume for a contradiction that $(R \cdot E)<0$, that is,\ $a <0$. Then $(R-aH)^2=(bE+c\eta)^2=-4c^2 \leq 0$. On the other hand $(R-aH)^2 = -2+a^2(4p+4)-2a(R \cdot H)>0$,
since $a<0$ and $H$ is nef. We have reached a contradiction, thus $E$ is nef.

\vskip 3pt

To conclude that $H$ is base point free, since $H$ is big and nef, it suffices to show that there is no smooth elliptic curve $F$ with $(H \cdot F)=1$ (cf. Proposition \ref{bpf-prop}). Suppose such an $F$  exists and write $F=aH+bE+c \eta$. Since $E$ is the class of an integral elliptic curve, $E$ is nef and $2a=(F \cdot E) \geq 0$. We have
$-4c^2=(F-aH)^2=a(-2+a(4p+4))$, which is only possible if $a=c=0$ and so $F=bE$. Since $(bE \cdot H)=2b$, this is a contradiction.
\end{proof}
\begin{lem} \label{lattice-1}
Let $\hat{Z}_g$ be a K3 surface with $\mathrm{Pic}(\hat{Z}_g)= \hat{\Xi}_{g,p}$, such that $H$ is big and nef. Then no class of the form $aE+b\eta$, for $b \neq 0$, can be effective.
\end{lem}
\begin{proof}
Suppose $aE+b\eta$ is effective, where $b \neq 0$. Since $(aE+b\eta)^2=-4b^2$, there must be an integral component $R$ of $aE+b\eta$ with $(R \cdot aE+b\eta)<0$. Since $R$ is integral and not nef, we must have $(R)^2=-2$. Since $(aE+b \eta \cdot E)=0$ and $E$ is nef, we have $(R \cdot E)=0$ and thus $R$ is of the form $xE+y \eta$ for $x,y \in \mathbb{Z}$. But then $(R)^2=-4y^2 \neq -2$, which is a contradiction.
\end{proof}
\begin{lem}
Let $\hat{Z}_g$ be as above. Then any class $A:=H+c\eta$ with $c \in \mathbb{Z}$ and satisfying $(A)^2 \geq 0$ is nef.
\end{lem}
\begin{proof}
Suppose by contradiction that $A$ is not nef. As $(A)^2 \geq 0$ and $(A \cdot E)=2 >0$, the class $A$ is effective. Thus there exists an integral base component $R$ of $A$
with $(R)^2=-2$ and $(R \cdot A) <0$. Write $R=xH+yE+z\eta$ for $x,y,z \in \mathbb{Z}$; as $(R)^2=-2$, we obtain $x \neq 0$. Since $A-R$ is effective ($R$ is
a base component of $A$), intersecting with $E$ gives $x=1$ and $A-R=-yE+(c-z) \eta$. From Lemma \ref{lattice-1} we have $c=z$ and then $y \leq 0$. If $y=0$, then we would have $R=A$ which contradicts that $(A)^2 \geq 0$, so $y<0$. But now $R=H+c\eta+yE=A+yE$, so we write
$$-2=(R)^2=(R \cdot A) +y(E \cdot R) =(R \cdot A)+2y  \leq -3,
$$
for $(R \cdot A)<0$ and $y <0$. This is a contradiction.
\end{proof}
\begin{cor} \label{h1-cor}
Let $\hat{Z}_g$ be as above, and set $L=H-\eta$. Then $H^1(\hat{Z}_g, qH-L)=0$ for $q \geq 0$.
\end{cor}
\begin{proof}
For $q=0$, note that $(L)^2=2g-2 >0$, so $L$ is big and nef by the previous lemma and $H^1(\hat{Z}_g, -L)=0$. For $q=1$, note that $(\eta)^2=-4$ and that neither $\eta$ nor $-\eta$ are effective by Lemma \ref{lattice-1}. Thus $H^1(\hat{Z}_g, \eta)=0$. For the remaining cases, it suffices to show
$2H-L=H+\eta$ is big and nef, since $H$ is big and nef. We have $(H+\eta)^2=4p+2(2(p-i)+1)>0$.
\end{proof}
We have seen that the line bundle $L$ is big and nef. We now show that it is base point free.
\begin{lem}
For $\hat{Z}_g$ as above, the class $L=H-\eta$ is base point free.
\end{lem}
\begin{proof}
As $L$ is big and nef, it suffices to show that there is no smooth elliptic curve $F$ with $(L \cdot F)=1$. Suppose such an $F$ were to
exist, and write $F=aH+bE+c\eta$. As $(F)^2=0$, we obtain that $a \neq 0$. As $(F \cdot E) \geq 0$, we find $a >0$. We calculate
$$
-4(c+a)^2=(F-a(H-\eta))^2
=a^2(2g-2)-2a =a(a(2g-2)-2).
$$
As $a>0$ and $2g-4>0$, this is a contradiction.
\end{proof}

\begin{lem}
For $\hat{Z}_g$ as above, the class $B:=H-(2p+2-j)E-\eta$ is not effective for $j\leq p$.
\end{lem}
\begin{proof}
We calculate $(B)^2= 4(i+j)-8p-10 \leq -6$. If $B$ is effective, then there exists a base component $R$ of $B$ with $(R \cdot B)<0$ and $(R)^2=-2$. Write
$R=aH+bE+c \eta$ for $a,b,c \in \mathbb{Z}$. Since $R$ is effective, $(R \cdot E) \geq 0$, so $a \geq 0$. We have $a \neq 0$ because $(R)^2=-2$. Since $B-R$ is effective, $(B-R \cdot E) \geq 0$ which forces $a=1$. Then $B-R=-(2p+2-j+b)E-(1+c)\eta$. Applying Lemma \ref{lattice-1}, we see $c=-1$, and we have $2p+2-j+b \leq 0$, so $b \leq -p-2$. But then $R=H+bE-\eta$ and one calculates
$$
 (R)^2=4(b+i)-2  \leq 4(i-p-2) -2 \leq -6,
$$
 which is a contradiction.
\end{proof}
\begin{cor} \label{kernel-bundle-cor}
Let $\hat{Z}_g$ be as above and let $D \in |H|$ be an integral, smooth curve. Then $H^1\Bigl(D, \bigwedge^j M_{K_D}(K_D+\eta_D)\Bigr)=0$ for $j \leq p$.
\end{cor}
\begin{proof} Essentially identical to that of Lemma \ref{hypvanodd}.
\end{proof}

\begin{thm}\label{evenkoszul}
Let $\hat{Z}_g$ be as above and $Z_g$ be a generic $\Xi_{g,p}$-polarized K3 surface, which is deformation equivalent to $\hat{Z_g}$. Let $C \in |L|$ be a smooth, integral curve, where $L=H-\eta$. Then $K_{j,2}(C,H_C)=0 \text{ for $j \leq p$.}$
\end{thm}
\begin{proof}
On the surface $Z_g$, we choose a general divisor $D' \in |H|$.
By semicontinuity and Corollary \ref{kernel-bundle-cor}, we have $H^1\Bigl(D', \bigwedge^j M_{K_{D'}}(K_{D'}+\eta)\Bigr)=0$, for $j \leq p$. Thus $K_{j-1,3}(Z_g,-L,H)=0$ for $j \leq p$. We further have $h^1(Z_g, qH-L)=0$ for $q \geq 0$ from Corollary \ref{h1-cor} and semicontinuity, as well as $h^0(Z_g, H-L)=0$ from Lemma \ref{lattice-1}. Thus Proposition \ref{prop1} applies, and it suffices to show $K_{j,2}(Z_g,H)=0$ for $j \leq p$. For this it suffices to show $\text{Cliff}(D') \geq p+1$ by \cite[Thm.\ 1.3]{aprodu-farkas-green}.  To establish this, it suffices in turn to show there is no decomposition $H=A_1+A_2$ for divisors $A_i$ on $Z_g$ with $h^0(Z_g, A_i) \geq 2$ for $i=1,2$.

Suppose $H=A_1+A_2$ is such a decomposition and write $A_i=a_iH+b_i\eta$ for $a_i, b_i \in \mathbb{Z}$, when $i=1,2$. By semicontinuity, $A_i$ must deform to effective divisors on $\hat{Z}_g$, and then intersecting with $E$ shows that $a_i \geq 0$ for $i=1,2$. Since $a_1+a_2=1$, we have either $a_1=0$ or $a_2=0$. We assume $a_1=0$, so $A_1=b_i \eta$ is effective. By semicontinuity and Lemma \ref{lattice-1}, we see $b_i=0$, so $A_1$ is trivial and the claim holds.
\end{proof}
We thus derive the main result of this section:
\begin{thm}\label{paros}
Let $C$ be a general curve of genus $g=2i\geq 4$ and $H_C \in \mathrm{Pic}(C)$ be a general line bundle of degree $2p+2i+3,$ where $p+1\geq i$. Then $$K_{j,2}(C,H_C)=0 \text{ for $j \leq p$.}$$
\end{thm}
\begin{proof}
This follows from the above by setting $H_C:=H_{|C}$. Note that we have $h^0(C,H_C)=h^0(Z_g,H)=2p+4$ (as $h^0(Z_g,\eta)=h^1(Z_g, \eta)=0$ by deforming to $\hat{Z}_g$), which is the expected number of sections for a line bundle of degree $2p+2i+3$.
\end{proof}

\section{The Prym--Green conjecture for curves of odd genus}

In this section we prove the odd genus case of the Prym-Green Conjecture formulated in \cite{chiodo-eisenbud-farkas-schreyer}.
We start by recalling a few things about polarized Nikulin
surfaces.

The \emph{Nikulin lattice} $\mathfrak{N}$ is the even lattice of rank $8$ generated by elements $N_1, \ldots, N_8$ and $\mathfrak{e}:=\frac{1}{2}\sum_{j=1}^8 N_j $, where $N_j^2=-2$ for $j=1, \ldots, 8$ and $(N_{\ell}\cdot N_j)=0$ for $\ell \neq j$. For an integer $g\geq 2$, following  \cite[Definition \ 2.1]{gar-sarti-even} and \cite[Definition 1.2]{farkas-verra-moduli-theta}, we define a \emph{Nikulin surface of the first kind} to be a K3 surface $X$ together with a  primitive embedding $\mathbb{Z}\cdot L \oplus \mathfrak{N} \hookrightarrow \mathrm{Pic}(X)$ such that $L$ is a big and nef class with $(L)^2 =2g-2$.

\vskip 3pt

Such surfaces can be realised as the desingularisation of the quotient of a smooth K3 surface by a symplectic involution, see \cite[Remark p.\ 9]{gar-sarti-even} (note however that not all such quotients are Nikulin surfaces of the first kind when $g$ is odd). By definition, a Nikulin surface $X$ contains an even set of eight disjoint smooth rational curves $N_1, \ldots, N_8$ which generate $\mathfrak{N}$ over $\frac{1}{2}\mathbb{Z}$ (but not over $\mathbb{Z}$).
Nikulin surfaces of the first kind form an irreducible $11$-dimensional moduli space $\mathcal{F}_g^{\mathfrak{N}}$. We refer to \cite{dolgachev} and \cite{vGS} for the construction of $\mathcal{F}_g^{\mathfrak{N}}$ via period domains and to \cite{farkas-verra-moduli-theta} for a description of its birational geometry for small genus.

A general element $\bigl[X, \mathbb Z\cdot L \oplus \mathfrak{N}\hookrightarrow \mbox{Pic}(X)\bigr]\in \mathcal{F}_g^{\mathfrak{N}}$ corresponds to a surface having Picard lattice equal to $\Lambda_{g}:=\mathbb{Z}\cdot L \oplus \mathfrak{N}$. Set  $H=L-\mathfrak{e}\in \mbox{Pic}(X)$, where $2\mathfrak{e}=N_1+ \ldots+ N_8$. For a smooth curve $C \in |L|$, since $\mathfrak{e}_C^{\otimes 2}=0$, the pair $[C, \mathfrak{e}_C]$ is an element of the Prym moduli space $\cR_g$ and
$$\phi_{H_C}:C\rightarrow \PP^{g-2}$$ is a Prym-canonical curve of genus $g$.

\vskip 3pt

Suppose now that $g=2i+3$ is an odd genus. The Prym-Green Conjecture (\ref{pgparatlan})
predicts that
$$ K_{i,1}(C,H_C)=0 \ \mbox{ and } \ K_{i-2,2}(C,H_C)=0.$$
This is equivalent to determining the shape of the resolution the Prym-canonical curve $\phi_{H_C}(C)$.

\begin{lem} \label{pg-odd}
Let $g=2i+3$ for $i \geq 2$. There exists a Nikulin surface $X_{g}$ with $\mathrm{Pic}(X_{g}) \simeq  \Lambda_{g}$ such that $L$ is base point free, $H$ is very ample and the general smooth curve $D \in |H|$ is Brill-Noether general.
\end{lem}
\begin{proof}
Consider a general Nikulin surface $X_g$ of the first kind and let $\pi:X_g \to \bar{Y}$ be the map which contracts all the exceptional curves $N_j$. The base point freeness of $L$ is a consequence of \cite[Proposition 3.1]{gar-sarti-even}. Since $g\geq 6$, from \cite[Lemma 3.1]{gar-sarti-even} we obtain that $H=L-\mathfrak{e}$ is very ample.
The fact that a general smooth curve $D\in|H|$ is Brill--Noether general relies once more on showing that there is no decomposition $H=A_1+A_2$, for divisors $A_i$ on $X_g$ with $h^0(X_g, A_i) \geq 2$, for $i=1,2$. Suppose there is such a decomposition and write $A_i=a_i L+b_i\mathfrak{m}_i$, where $a_i, b_i\in \mathbb Z$ and $\mathfrak{m}_i \in \mathfrak{N}$,
for $i=1,2$. By intersecting with the nef class $L$, we have that $a_i \geq 0$ for $i=1,2$. Since $a_1+a_2=1$, we may assume $a_1=0$ and $a_2=1$; thus $A_1 \in \mathfrak{N}$ is orthogonal to $L$. This implies that $\pi(A_1)$ is a finite sum of points (as $\pi(L)$ is ample and in fact generates $\mbox{Pic}(\bar{Y})$). This in turn forces $A_1$ to be a sum of the disjoint $(-2)$ curves $N_1, \ldots, N_8$, so that $h^0(X_g, A_1)=1$.
\end{proof}

\begin{cor} \label{koszul-prym}
Let $X_g$ and $H$ be as above. Then $K_{i,1}(X_g,H)=K_{i-2,2}(X_g,H)=0$.
\end{cor}
\begin{proof}
Let $D \in |H|$ be a general divisor, hence $g(D)=2i+1$. From the previous lemma, $\mbox{Cliff}(D)=i$. Thus the result follows from \cite[Theorem 1.3]{aprodu-farkas-green}.
\end{proof}

\begin{lem} \label{coh-lemma}
Assume $g \geq 11$ is odd and let $X_g$ be a general Nikulin surface with Picard lattice $\Lambda_{g}$. Then $H^1(X_g,qH-L)=0$ for $q \geq 0$ and $H^0(X_g,H-L)=0$.
\end{lem}
\begin{proof}
As already pointed out, $H$ is very ample for $g \geq 6$. From \cite[Proposition 3.5]{gar-sarti-even},  the class $2H-L$ is big and base point free for $g \geq 10$; thus $H^1(X_g,qH-L)=0$ for $q \geq 2$. Furthermore, $H^0(X_g,H-L)=H^1(X_g,H-L)=0$ from \cite[Lemma  1.3]{farkas-verra-moduli-theta} and $H^1(X_g,L)=0$ since $L$ is big and nef.
\end{proof}

From the above results and Proposition \ref{prop1}, to establish the Prym-Green Conjecture for a general curve $C\in |L|$, it suffices to show that for a smooth curve $D\in |H|$
$$K_{i-1,2}(D,-L_D,K_D)=0 \ \mbox{ and } \ K_{i-3,3}(D,-L_D,K_D)=0.$$

Observe that both these statements would follow from the Minimal Resolution Conjecture \cite{farkas-popa-mustata}, provided  $L_{D}\in \mbox{Pic}(D)$ was a general line bundle in its respective Jacobian, which is of course not the case. However, one can try to follow the proof of \cite[Lemma 3.3]{farkas-popa-mustata} and specialize to hyperelliptic curves while still retaining the embedding of the Prym curve $C$ in a Nikulin surface. We carry this approach out below.

Let $\mathfrak{T}_g$ be the rank two lattice with ordered basis $\{ L,E\}$ and with intersection form
\[ \left( \begin{array}{cc}
2g-2 & 2  \\
2 & 0  \end{array} \right).\]
Consider the lattice $\mathfrak{T}_g \oplus \mathfrak{N}$, where $\mathfrak{N} $ is the Nikulin lattice. This rank $10$ lattice is even,
of signature $(1,9)$; thus by a result of Nikulin and the Torelli theorem \cite{dolgachev}, there exists a $K3$ surface $\hat{X}_g$ with
Picard lattice isomorphic to  $\mathfrak{T} \oplus \mathfrak{N} $. After applying Picard-Lefschetz reflections and replacing $L$ with $-L$ if necessary, we may assume $L$ is nef. Since we have a primitive embedding $\Lambda_g \hookrightarrow\mathfrak{T}_g \oplus \mathfrak{N}$, we conclude that $\hat{X}_g$ is a Nikulin surface of the first kind.
\begin{lem}
Let $\hat{X}_g$ be as above with $g \geq 11$. Then $H=L-\mathfrak{e}\in \mathrm{Pic}(\hat{X}_g)$ is big and nef.
\end{lem}
\begin{proof}
As $(H)^2>0$, we need to show that $H$ is nef. The class $H$ is effective, as $(H \cdot L)>0$. Suppose for a contradiction that there exists a  $(-2)$ curve $\Gamma$ with $(\Gamma \cdot H)<0$. We may write $$\Gamma=aL+bE+\sum_{j=1}^8 c_j N_j,$$
for $a,b \in \mathbb{Z}$ and $c_j \in \frac{1}{2}\mathbb{Z}$ for $j=1, \ldots, 8$. Since $(N_j \cdot H)=1$ for all $j$, we have $\Gamma \neq N_j$, so that $(N_j \cdot \Gamma) \geq 0$ and hence $c_j \leq 0$ for all $j$. Set $k:=(\Gamma \cdot L)=a(2g-2)+2b$; since $L$ is big and nef, $k \geq 0$. From $( \Gamma \cdot H) <0$ we have $k < \sum_{j=1}^8 |c_j|$ and thus $k^2 < 8\sum_{j=1}^8 c_j^2$, by the Cauchy--Schwarz inequality. Further, since $(\Gamma)^2=-2$, we find that $(aL+bE)^2=2\sum_{j=1}^8 c_j^2-2$, and then the Hodge index theorem implies $(L)^2(aL+bE)^2 \leq \bigl(L \cdot (aL+bE)\bigr)^2$, or equivalently
$(2g-2)\bigl(2 \sum_{j=1}^8 c_j^2-2\bigr) \leq k^2$. Putting these two inequalities together, we obtain $\sum_{j=1}^8 c_j^2 < \frac{g-1}{g-3}$, and so $\sum_{j=1}^8 (2c_j)^2 \leq 4$, for $g \geq 11$.

Since $2c_j\in \mathbb{Z}$, there are two cases. In the first case, $c_{\ell}=-1$ for some $\ell$ and $c_j=0$ for $j \neq \ell$. From $k < \sum_{j=1}^8 |c_j|$, we then have $k=0$ and further  $(aL+bE)^2=0$. Putting these together gives $2ab=0$ and thus $a=0$ or $b=0$. Using again that $k=\bigl(L \cdot (aL+bE)\bigr)=0$, we must have $a=b=0$ so $\Gamma=-N_{\ell}$, which contradicts the effectiveness of $\Gamma$. In the second case $|c_j| \leq \frac{1}{2}$ for all $j$, and there are at most four values of $j$ such that $c_j\neq 0$. But since $\mathfrak{N}$ is generated by $N_1, \ldots, N_8$ and $\mathfrak{e}$ with $2\mathfrak{e}=N_1+ \ldots +N_8$, the only such elements in $\mathfrak{T}_g \oplus \mathfrak{N} $ must have $c_j=0$ for all $j$. But this implies $k<0$ which is a contradiction.
\end{proof}
\begin{lem}
Let $\hat{X}_g$ and $H$ be as above, with $g \geq 11$. Then $H$ is base point free.
\end{lem}
\begin{proof}
From Proposition \ref{bpf-prop} and the previous lemma, it suffices to show that there is no smooth elliptic curve $F$ with $(F \cdot H)=1$. Suppose such an $F$ exists and write $F=aL+bE+\sum_{j=1}^8 c_j N_j$ for $a,b \in \mathbb{Z}$ and $c_j \in \frac{1}{2}\mathbb{Z}$ for $j=1,\ldots, 8$. We have $c_j\leq 0$ since $(F \cdot N_j) \geq 0$.

Set $k:=(F \cdot L)=a(2g-2)+2b \geq 0$. Since $1=(F \cdot H)$, we find $k=1+ |c_1|+\cdots+|c_8|$. Note that
$(aL+bE)^2=2\sum_{j=1}^8 c_j^2.$ The Hodge Index Theorem applied to $L$ and $aL+bE$ gives
$$(2g-2)\Bigl(2 \sum_{j=1}^8 c_j^2\Bigr) \leq k^2 =\Bigl(1+\sum_{j=1}^8 |c_j|\Bigr)^2.$$ One has $(1+\sum_{i=1}^8 |c_i|)^2 \leq 1+\sum_{j=1}^8 2|c_j|+8\sum_{j=1}^8 c_j^2 \leq 1+12\sum_{j=1}^8 c_j^2$, using the fact that $2 c_j \in \mathbb{Z}$. Combining this with the above inequality, we write $\bigl(4(g-1)-12\bigr)\sum_{j=1}^8 c_j^2 \leq 1$ and hence $\sum_{j=1}^8 (2c_j)^2=0$ for $g \geq 6$. Thus $c_j=0$ for all $j$ and $k=1$. But $k=a(2g-2)+2b$ is even, which is a contradiction.
\end{proof}
We now show that the $E$ is the class of an irreducible, smooth elliptic curve.
\begin{lem}
Let $\hat{X}_g$ and $H$ be as above with $g \geq 11$. Then $E$ is the class of an irreducible, smooth elliptic curve.
\end{lem}
\begin{proof}
As $E$ is primitive with $(E)^2=0$, it suffices to show that $E$ is nef, \cite[Proposition 2.3.10]{huy-lec-k3}. The class $E$ is effective as $(E)^2=0$ and $(E \cdot L)>0$. If $E$ is not nef, then there exists a smooth rational curve $R$ with $(R \cdot E)<0$. Obviously, $R$ must be a component of the base locus of $|E|$. Write $R=aL+bE+\sum_{j=1}^8 c_j N_i$, for $a,b \in \mathbb{Z}$, $c_j \in \frac{1}{2} \mathbb{Z}$. From $(R \cdot E)<0$, we have $a<0$. Since $L$ is nef, $(R \cdot L) \geq 0$ which gives $b \geq -a(g-1)$. Further, $E-R$ is effective, so $\bigl(L \cdot (E-R)\bigr) \geq 0$, giving $b \leq 1-a(g-1)$.  So either $b=-a(g-1)$, or $b=1-a(g-1)$. In either case, $(R \cdot L)=2(a(g-1)+b) \leq 2$. The Hodge Index Theorem applied to $L$ and $aL+bE$ yields
$$(2g-2)\Bigl(2 \sum_{j=1}^8 c_j^2-2\Bigr) \leq (R \cdot L)^2 \leq 4. $$
Thus $\sum_{j=1}^8 (2c_j)^2 \leq 4$ for $g \geq 6$, as $2c_j \in \mathbb{Z}$ for all $j$. Since $\mathfrak{N}$ is generated by $N_1, \ldots, N_8$, and $\mathfrak{e}$, the integers $2c_j$ all have the same parity for $j=1, \ldots, 8$, and so there are only two possibilities, namely $c_j=0$, for all $j$, or there is an index $\ell$ with $c_{\ell}=-1$ and $c_j=0$, for $j\neq \ell$ (intersecting with $N_j$, shows that $c_j \leq 0$ for all $j$ as in the previous lemmas). In the case $c_1=\ldots=c_8=0$, we have either $R=a(L-(g-1)E)$ or $R=aL+(1-a(g-1))E$. Then $(R)^2=-2$ leads to either $a^2(g-1)=1$ which is obviously impossible, or $1=a(a(g-1)-2)$, which is also not possible for $a<0$. In the case that there exists $\ell$ with $c_\ell=1$ and $c_j=0$, for  $j \neq \ell$, we have either $R=a(L-(g-1)E)-N_j$ or $R=aL+(1-a(g-1))E-N_j$. Then $(R)^2=-2$ leads to either $a^2(g-1)=0$ or $2a(2-a(g-1))=0$, either of which forces $a=0$, contradicting $a<0$.
\end{proof}

As a consequence of the above, we see that any smooth curve $D \in |H|$ is hyperelliptic, with $\mathcal{O}_{D}(E)$ defining a pencil $\mathfrak g^1_2$. The next two technical lemmas will be needed later.
\begin{lem}
Let $\hat{X}_g$ and $H$ be as above with $g \geq 11$ odd. Then $L-\frac{g-1}{2}E$ is base point free.
\end{lem}
\begin{proof}
We have $(L-\frac{g-1}{2}E)^2=0$ and  $L \cdot (L-\frac{g-1}{2}E)=g-1>0$, so $L-\frac{g-1}{2}E$ is effective and it suffices to show that it is nef. Suppose by contradiction that there is a smooth rational curve $R$ with $R \cdot (L-\frac{g-1}{2}E)<0$ and write $R=aL+bE+\sum_{j=1}^8 c_j N_j$ as above. By intersecting with $N_j$, we see that  $c_j \leq 0$ for $j=1,\ldots, 8$. Next, since $R$ is a component of $L-\frac{g-1}{2}E$, the class $L-\frac{g-1}{2}E-R$ must be effective, and intersecting it with the nef class $E$ gives $a\leq 1$. Further, $R \cdot (L-\frac{g-1}{2}E)<0$ yet $(R \cdot L) \geq 0$, so $(R \cdot E)>0$ and thus $a=1$. Intersecting the effective class $L-\frac{g-1}{2}E-R$ with the nef class $L$ now yields $b \leq -\frac{g-1}{2}$. Hence $(L+bE)^2=2g-2+4b \leq 0$. From $(R)^2=-2$ we deduce $(L+bE)^2=2 \sum_{j=1}^8 c_j^2-2$. Thus $\sum_{j=1}^8 (2c_j)^2 \leq 4$. Using that, as in the previous lemma $2c_j \in \mathbb{Z}$ have the same parity, we distinguish two possibilities: either all $c_j=0$, or else, there exists $\ell$ with $c_{\ell}=-1$ and $c_j \neq 0$ for $j \neq \ell$.

In the former case, $R=L+bE$, with $b \leq -\frac{g-1}{2}$, and $(R)^2=-2$ gives $-2=2(g-1)+4b$. As $g$ is odd, $2(g-1)+4b$ is divisible by $4$, a contradiction. In the latter case $R=L+bE-N_{\ell}$ and $(R)^2=-2$ implies $-2=2(g-1)+4b-2$, which produces $b=-\frac{g-1}{2}$. But then $(L-\frac{g-1}{2}E) \cdot R=(L-\frac{g-1}{2}E) \cdot (L-\frac{g-1}{2}E-N_{\ell})=0$, contradicting the assumptions.
\end{proof}
\begin{lem}
Let $\hat{X}_g$ and $H$ be as above with $g \geq 11$ odd. Then the class  $cE-\mathfrak{e}\in \mathrm{Pic}(\hat{X}_g)$ is not effective for any $c \in \mathbb{Z}$.
\end{lem}
\begin{proof}
By intersecting with $L$, we see that $cE-\mathfrak{e}$ is not effective if $c <0$. Suppose there exists an integer $c>0$ such that $cE-\mathfrak{e}$ is effective and we choose $c$ minimal with this property. Since $(cE-\mathfrak{e})^2=-4$, there is an integral component $R$ of $cE-\mathfrak{e}$ with $(R \cdot c E-\mathfrak{e})<0$. Necessarily, $(R)^2=-2$. Write $R=aL+bE+\sum_{j=1}^8 c_j N_j$, as above. We have $(N_j \cdot cE-\mathfrak{e})=1$, so $R \neq N_j$ and $(R \cdot N_j) \geq 0$, implying $c_j \leq 0$, for $j=1, \ldots,  8$. Intersecting $R$ with the nef class $E$ yields $a \geq 0$. Since $R$ is a component of $cE-\mathfrak{e}$, we have that $cE-\mathfrak{e}-R$ is an effective class which we intersect  with $E$, forcing $a=0$. From $(R)^2=-2$, we have $\sum_{j=1}^8 c_j^2=1$. As the integers $2c_j$ all have the same parity, the only possibility is that there exists $\ell$ with $c_{\ell}=-1$ and $c_j=0$ for $j \neq \ell$. Then $R=bE-N_{\ell}$, and intersecting with $L$ shows $b \geq 0$. We have $b>0$, for $-N_{\ell}$ is not effective. But then $cE-\mathfrak{e}-R=(c-b)E-\mathfrak{e}+N_{\ell}$ is effective.  Since $N_{\ell} \cdot ((c-b)E-\mathfrak{e}+N_{\ell})=-1<0$, we necessarily have that $N_{\ell}$ is a component of $(c-b)E-\mathfrak{e}+N_{\ell}$, so that $(c-b)E-\mathfrak{e}$ is effective. This contradicts the minimality of $c$.
\end{proof}

We are now in a position to show that for the hyperelliptic Nikulin surface $\hat{X}_g$ constructed before, the vanishing statements (\ref{pgparatlan}) hold.

\begin{cor}
Let $\hat{X}_g$ and $H$ be as above with $g=2i+3 \geq 11$ and $D \in |H|$ be smooth and irreducible. Then $H^0\Bigl(D,\bigwedge^{i-1} M_{K_D}\otimes K_D^{\otimes 2}\otimes L_D^{\vee}\Bigr)=0$ and $H^1\Bigl(D,\bigwedge^{i-2} M_{K_D}\otimes K_D^{\otimes 2}\otimes L_D^{\vee}\Bigr)=0$.
\end{cor}
\begin{proof}
For the first vanishing,  we need to show that $H^0\Bigl(D, \bigwedge^{i-1} M_{K_D}\otimes K_D\otimes \mathfrak{e}_D^{\vee}\Bigr)=0$. As in \cite[Prop.\ 3.5]{farkas-popa-mustata}, the kernel bundle $M_{K_D}$ splits into a direct sum of line bundles and
we have $$\bigwedge^{i-1} M_{K_D} \simeq \mathcal{O}_D\Bigl(-(i-1)E_D\Bigr)^{\oplus \binom{g-3}{i-1}}.$$ Using that $K_D \simeq (g-3)E_{D}$, it suffices to show $h^0(D,\mathcal{O}_D((i+1)E_D-\mathfrak{e}_D))=0$. From the above lemmas, $h^0(\hat{X}_g,\mathcal{O}_{\hat{X}_g}((i+1)E-\mathfrak{e}))=0$ and the primitive class $L-(i+1)E$ is base point free, that is, it is represented by a smooth elliptic curve. Thus $$H^1\bigl(\hat{X}_g,\mathcal{O}_{\hat{X}_g}((i+1)E-L)\bigr)=H^1\bigr(\hat{X}_g,\mathcal{O}_{\hat{X}_g}((i+1)E-\mathfrak{e}-H)\bigr)=0$$
and the claim follows.

The second vanishing boils down to $h^1\bigl(D, \mathcal{O}_D((i+2)E_D-\mathfrak{e}_D)\bigr)=0$. The class $(i+2)E-\mathfrak{e}$ has self-intersection $-4$ and satisfies $H^0(\hat{X}_g, \mathcal{O}_{\hat{X}_g}((i+2)E-\mathfrak{e})\bigr)=0$ by the above lemma. Moreover $H^2\bigl(\hat{X}_g, \mathcal{O}_{\hat{X}_g}((i+2)E-\mathfrak{e})\bigr)=0$, since $\mathfrak{e}-(i+2)E$ is not effective, having negative intersection with $L$.  Also $H^1\bigl(\hat{X}_g, \mathcal{O}_{\hat{X}_g}((i+2)E-\mathfrak{e})\bigr)=0$ and $H^2\bigl(\hat{X}_g,\mathcal{O}_{\hat{X}_g}((i+2)E-\mathfrak{e}-H)\bigr)=0$, since $H-(i+2)E+\mathfrak{e}$ is not effective. Indeed, otherwise the smooth elliptic curve $E$ would be a subdivisor of $L-(i+1)E$, which is also the class of a smooth elliptic curve.
\end{proof}

\noindent \emph{Proof of Theorem \ref{hypnik}.}
By deformation to the hyperelliptic case $\hat{X}_g$, it then follows that there is a nonempty, open subset of the moduli space $\F_g^{\mathfrak{N}}$ of Nikulin surfaces $X_g$, such that  $$H^0\Bigl(D,\bigwedge^{i-1} M_{K_D}\otimes K_D^{\otimes 2}\otimes L_D^{\vee}\Bigr)=0 \ \mbox{ and } \ H^1\Bigl(D,\bigwedge^{i-2} M_{K_D}\otimes K_D^{\otimes 2}\otimes L_D^{\vee}\Bigr)=0,$$ for a general $D \in |H|$.
 In particular, $K_{i-1,2}(D,-L_D,K_D)=0$ and $K_{i-3,3}(D,-L_D,K_D)=0$. Putting everything together, we obtain $K_{\frac{g-3}{2},1}(C,K_C \otimes \mathfrak{e}_C)=K_{\frac{g-7}{2},2}(C,K_C\otimes \mathfrak{e}_C)=0$, for a general curve $C \in |L|$ general on such $X_g$. \hfill $\Box$

\section{The syzygy divisor on $\cM_{g,2g}$ and the divisorial case of the Green-Lazarsfeld Conjecture}

The goal of this section is to prove Theorem \ref{osztaly}. We use the convention that if $\ms$ is a Deligne-Mumford stack, then we denote by $\cM$ its coarse moduli space. All Picard and Chow groups of stacks and coarse moduli spaces considered in this section are with rational coefficients. We recall \cite{AC} that the coarsening map $\ms_{g,n} \rightarrow \mm_{g,n}$ induces an isomorphism between (rational) Picard groups $\mbox{Pic}(\mm_{g,n})\stackrel{\cong}\rightarrow \mbox{Pic}(\ms_{g,n})$. In particular, we shall stop distinguishing between divisor classes on $\ms_{g,n}$ and on $\mm_{g,n}$ respectively.
\vskip 3pt

 Recall that we have set $g=2i+1$ and $d=2g$ and defined $\mathfrak{Syz}$ to be the divisor in $\cM_{g,2g}$  of pointed curves $[C, x_1, \ldots, x_{2g}]$ such that
$$K_{i-1,2}\Bigl(C, \OO_C(x_1+\cdots+x_{2g})\Bigr)\neq 0.$$ To ease notation, we set $L:=\OO_C(x_1+\cdots+x_{2g})$. We first express determinantally over moduli the condition that a point belong to the divisor $\mathfrak{Syz}$. Let $M_{\PP^g}:=\Omega_{\PP^g}^1(1)$ be the universal rank $g$ kernel bundle over $\PP^g$. Then $K_{i-1,2}(C,L)\neq 0$ if and only if $H^0\bigl(\PP^{g}, \bigwedge^{i-1} M_{\PP^g}\otimes \mathcal{I}_{C/\PP^g}(2)\bigr)\neq 0$, or equivalently, the morphism between the following vector spaces of the same dimension
$$\varphi_{[C,L]}: H^0\Bigl(\PP^g, \bigwedge^{i-1} M_{\PP^g}(2)\Bigr)\longrightarrow H^0\Bigl(C, \bigwedge^{i-1} M_L\otimes L^{\otimes 2}\Bigr)$$
is an isomorphism. The statement that the two vector spaces above have the same dimension, follows because on one hand, it is well-known that
$$h^0\Bigl(\PP^{g}, \bigwedge^{i-1} M_{\PP^g}(2)\bigr)=h^0\Bigl(\PP^{g}, \Omega_{\PP^g}^{i-1}(i+1)\Bigr)=i{2i+3\choose i},$$
on the other hand, it is also known \cite{lazarsfeld-kernel} that $M_L$ is a stable vector bundle on $C$, hence
$H^1\bigl(C, \bigwedge^{i-1} M_L \otimes L^{\otimes 2}\bigr)=0$, for slope reasons; since $\mu(M_L)=-2$, one computes that
$$h^0\Bigl(C, \bigwedge^{i-1} M_L\otimes L^{\otimes 2}\Bigr)= {2i+1\choose i-1}\Bigl((i-1)\mu(M_L)+2d+1-g\Bigr)=i{2i+3\choose i+1}.$$
To express $\mathfrak{Syz}$ as a degeneracy locus of two vector bundles over the stack $\textbf{M}_{g,2g}$, we first introduce the following diagram of moduli stacks of pointed curves

\[
\begin{CD}
{\mathcal X}@>{v}>> \textbf{M}_{g,2g}\\
@VV{f}V@VV{\pi}V\\
{\textbf{C}}_g^{}@>{u}>>\textbf{M}_{g}\\
\end{CD}
\]
where $u:\textbf{C}_g\rightarrow \textbf{M}_g$ is the universal curve.
For $j=1,\ldots, 2g$, let
$q_j:\textbf{M}_{g,2g}\longrightarrow {\mathcal X}$  be the section of the universal family $v$ defined by
$q_j([C, x_1, \ldots, x_{2g}]):=\bigl([C, x_1, \ldots, x_{2g}], x_j\bigr)$. We set $E_j:={\rm Im}(q_j)$.
For $1\leq j\leq 2g$, we denote by $\psi_j\in \mbox{Pic}(\textbf{M}_{g,2g})$ the cotangent class corresponding to the $j$-th marked point, that is, characterized by fibres
$\psi_j\bigl([C,x_1, \ldots,x_{2g}]\bigr):=T_{x_j}^{\vee}(C)$. Finally, $\lambda:=c_1(u_*(\omega_u))\in \mbox{Pic}(\textbf{M}_g)$ denotes the Hodge class.
\vskip 3pt

For $\ell\geq 1$, we set $\mathbb F_{\ell}:=v_*\bigl(\OO_{\cX}(\ell E_1+\cdots+\ell E_{2g})\bigr)$. By Grauert's Theorem it follows that $\mathbb F_{\ell}$ is a vector bundle of rank $(2i+1)(2\ell-1)+1$. We define the kernel vector bundle over $\mathcal{X}$ via the evaluation  sequence:
$$0\longrightarrow \cM\longrightarrow v^*\mathbb F_{1}\longrightarrow \OO_{\cX}(E_1+\cdots+E_{2g})\longrightarrow 0.$$
Clearly,
$\cM_{| v^{-1}([C, x_1, \ldots, x_{2g}])}=M_L$, where we recall that we have set $L=\OO_C(x_1+\cdots+x_{2g})$. Next, for integers $p\geq 0$ and $q\geq 2$, we introduce the vector bundle
$$\G_{p,q}:=v_*\Bigl(\bigwedge^p \cM\otimes \OO_{\cX}(qE_1+\cdots +qE_{2g})\Bigr).$$
Observe that $\G_{0,q}=\mathbb F_q$. Using the vanishing $H^1\bigl(C, \bigwedge^p M_L\otimes L^{\otimes q}\bigr)=0$ valid for each line bundle $L\in \mbox{Pic}^{2g}(C)$ and integer $q\geq 2$, we conclude that $\G_{p,q}$ is locally free over $\textbf{M}_{g,2g}$. Furthermore, there are exact sequences of vector bundles
\begin{equation}\label{seqg}
0\longrightarrow \G_{p,q}\longrightarrow \bigwedge ^p\G_{0,1}\otimes \G_{0,q}\longrightarrow \G_{p-1,q+1}\longrightarrow 0,
\end{equation}
globalizing the corresponding exact sequences at the level of each individual curve. Following the path indicated in \cite{Fa}, for each $p\geq 0$ and $q\geq 2$, we define a vector bundle $\H_{p,q}$ over $\textbf{M}_{g,2g}$ such that
$$\H_{p,q}\bigl([C, x_1, \ldots, x_{2g}]\bigr):=H^0\Bigl(\PP^{g}, \bigwedge^p M_{\PP^g}(q)\Bigr),$$
where the last identification takes into account the embedding $\phi_L:C\hookrightarrow \PP^g$. To define the bundles $\H_{p,q}$, we proceed inductively. First, we set $\H_{0,1}:=\G_{0,1}$ and then $\H_{0,q}:=\mbox{Sym}^q \ \H_{0,1}$. Then, having defined $\H_{j,q}$ for all $j<p$, we define $\H_{p,q}$ via the following exact sequence:
\begin{equation}\label{seqh}
0\longrightarrow \H_{p,q}\longrightarrow \bigwedge ^p\H_{0,1}\otimes \H_{0,q}\longrightarrow \H_{p-1,q+1}\longrightarrow 0.
\end{equation}
There exist vector bundle morphism $\varphi_{p,q}: \H_{p,q}\rightarrow \G_{p,q}$, which over each fibre corresponding to an embedding $C\stackrel{|L|}\hookrightarrow \PP^g$ are the restriction maps at the level of twisted holomorphic forms:
 $$H^0\Bigl(\PP^g, \bigwedge^p M_{\PP^g}(q)\Bigr)\rightarrow H^0\Bigl(C, \bigwedge^p M_L\otimes L^{\otimes q}\Bigr).$$
 Note that $\mbox{rk}(\H_{i-1,2})=\mbox{rk}(\G_{i-1,2})$ and
$\varphi=\varphi_{i-1,2}:\H_{i-1,2}\rightarrow \G_{i-1,2}$ is the morphism whose degeneracy locus is precisely the divisor $\mathfrak{Syz}$.
The following formulas are standard, see \cite{HM}:

\begin{lem}\label{szabalyok}
Keeping the notation from above, the following identities hold:
\begin{enumerate}
\item {\rm (i)} $v_*(f^*c_1(\omega_u)^2)=12\lambda$.
\item {\rm (ii)} $v_*(v^*\lambda\cdot f^*c_1(\omega_u))=(2g-2)\lambda$.
\item {\rm (iii)} $v_*([E_j]\cdot v^*\lambda)=\lambda$.
\item {\rm (iv)} $v_*([E_j]\cdot f^*c_1(\omega_u))=\psi_j$, for $1\leq j\leq 2g$.
\item {\rm (v)} $v_*([E_j]^2)=-\psi_j$, for $1\leq j\leq 2g$.
\end{enumerate}
\end{lem}

Using the exact sequence (\ref{seqg}) and (\ref{seqh}), we shall reduce the calculation of the first Chern class of $\G_{p,q}$ and $\H_{p,q}$ respectively to the case $p=0$.  We have the following result in that case:

\begin{prop}\label{chernl}
For $\ell\geq 1$, the following relation holds in $CH^1(\textbf{M}_{g,2g})$:
$$c_1(\G_{0,\ell})=\lambda-{\ell+1\choose 2}\sum_{j=1}^{2g} \psi_j.$$
\end{prop}
\begin{proof}We apply Grothendieck-Riemann-Roch to the proper morphism of stacks $v:\cX\rightarrow \textbf{M}_{g,2g}$ and to the sheaf $\L:= \OO_{\cX}(E_1+\cdots+E_{2g})$. Note that $R^1v_*(\L^{\otimes \ell})=0$, therefore we can write:
$$c_1(\G_{0, \ell})=v_*\Bigl[\Bigl(1+\ell\ \sum_{j=1}^{2g} [E_j]+\frac{\ell^2}{2}\Bigl(\sum_{j=1}^{2g} [E_j]\Bigr)^2+\cdots\Bigr)\cdot \Bigl(1-\frac{f^*c_1(\omega_u)}{2}+\frac{f^*c_1^2(\omega_u)}{12}+\cdots\Bigr)\Bigr]_2.$$
Applying repeatedly Lemma \ref{szabalyok}, we get the claimed formula.
\end{proof}

We now  prove Theorem \ref{osztaly}. The morphism $\varphi:\H_{i-1,2}\rightarrow \G_{i-1,2}$ degenerates along $\mathfrak{Syz}$. In the course of proving Theorem \ref{secgen}, we have exhibited a pair $[C,L]\in \mathfrak{Pic}^{2g}_g$ with $K_{i-1,2}(C,L)=0$. Therefore $\varphi$ is generically non-degenerate.
In the next proof, we shall use the formulas
$$c_1\bigl(\mbox{Sym}^n(E)\bigr)={r+n-1\choose r}c_1(E) \ \mbox{ and } \ c_1\bigl(\bigwedge^n E\bigr)={r-1\choose n-1}c_1(E)$$ valid for a vector bundle $E$ of rank $r$ on any stack or variety.

\vskip 5pt

\noindent \emph{Proof of Theorem \ref{osztaly}.} From the discussion above, $[\mathfrak{Syz}]=c_1(\G_{i-1,2}-\H_{i-1,2})\in CH^1(\textbf{M}_{2g,g})$. We compute both  Chern classes via the exact sequences (\ref{seqg}) and (\ref{seqh}) respectively, coupled with Proposition \ref{chernl} and write
$$c_1(\G_{i-1,2})=\sum_{\ell=0}^{i-1} (-1)^{\ell}\ c_1\Bigl(\bigwedge^{i-\ell-1} \G_{0,1}\otimes \G_{0, \ell+2}\Bigr)$$
$$=\sum_{\ell=0}^{i-1} (-1)^{\ell}\Bigl[ \mbox{rk}(\G_{0, \ell+2})\cdot{2i+1\choose i-\ell-2} c_1(\G_{0,1})+ {2i+2\choose i-\ell-1} c_1(\G_{0, \ell+2})\Bigr]$$
$$=\sum_{\ell=0}^{i-1} (-1)^{\ell}\Bigl[{2i+1\choose i-\ell-2}\Bigl((2i+1)(2\ell+3)+1\Bigr)c_1(\G_{0,1})+{2i+2\choose i-\ell-1}\Bigl(\lambda-{\ell+3\choose 2}\sum_{j=1}^{2g}\psi_j\Bigr)\Bigr]=$$
$$={2i \choose i}\Bigl(\frac{4i^3+5i^2-4i-2}{(i+1)(i+2)}\cdot \lambda - \frac{8i^3+13i^2-i-2}{2(i+1)(i+2)}\cdot \sum_{j=1}^{2g} \psi_j\Bigr).$$
On the other hand, recalling that $\H_{0, \ell+2}=\mbox{Sym }^{\ell+2}\H_{0,1}$ and that $\H_{0,1}=\G_{0,1}$, we have
$$c_1(\H_{i-1,2})=\sum_{\ell=0}^{i-1} (-1)^{\ell}\ c_1\Bigl(\bigwedge^{i-\ell-1} \H_{0,1}\otimes \H_{0, \ell+2}\Bigr)=$$$$\sum_{\ell=0}^{i-1}(-1)^{\ell}\Bigl[{2i+2\choose i-\ell-1}c_1(\H_{0, \ell+2})+{2i+1\choose i-\ell-2}\cdot \mbox{rk}(\H_{0, \ell+2})\cdot c_1(\H_{0,1})\Bigr]$$
$$=\sum_{\ell=0}^{i-1} (-1)^{\ell}\Bigl[{2i+2\choose i-\ell-1}{2i+3+\ell\choose \ell+1}c_1(\H_{0,1})+{2i+1 \choose i-\ell-2}{2i+3+\ell\choose \ell+2}c_1(\H_{0,1})\Bigr]
$$
$$=
{2i\choose i} \frac{i(2i+1)(2i+3)}{(i+1)(i+2)} \cdot \Bigl(\lambda - \sum_{j=1}^{2g} \psi_j \Bigr).$$
Computing the difference $c_1(\G_{i-1,2}-\H_{i-1,2})$, we obtained the claimed formula for the class $[\mathfrak{Syz}]$. \hfill $\Box$

\vskip 4pt

Finally, we explain how the divisorial case $d=2g$ of the G-L Secant Conjecture (Theorem \ref{diveset}) implies the conjecture for line bundles of extremal degree $d=2g+p+1-\mbox{Cliff}(C)$ on a  general curve $C$. The argument proceeds in two steps.

\vskip 5pt

\noindent \emph{Proof of Theorem \ref{hatareset} in the case $d\neq 2g+1$.}
We begin with the case of odd genus and take a smooth curve $C$ of genus $g=2i+1$ of Clifford index $i$, then set $d:=3i+p+3$, where $i\geq p+1$. In particular, inequality (\ref{ineq1}) is an equality, while inequality (\ref{ineq2}) is satisfied as well.

\vskip 3pt
 We start with a line bundle $L\in \mbox{Pic}^d(C)$  and let us assume that
$K_{p,2}(C,L)\neq 0$. We set $\xi:=L- K_C\in \mbox{Pic}^{p-i+3}(C)$. By using Lemma \ref{reduction-by-one},  for every effective divisor
$D\in C_{i-p-1}$, we obtain that $K_{i-1,2}(C, L(D))\neq 0$. Theorem \ref{diveset} implies $L+D\in K_C+C_{i+1}-C_{i-1}$, that is,
$$\xi+C_{i-p-1}\subset C_{i+1}-C_{i-1}.$$
This implies that $\mbox{dim } V_{g-p-3}^{g-p-4}(2K_C-L)\geq i-p-1$, which is contradiction, hence $K_{p,2}(C,L)=0$.

\vskip 4pt

Assume now that $C$ is a curve of even genus $g=2i$ and $\mbox{Cliff}(C)=i-1$. We grant Theorem \ref{even}. Out of this, we derive Theorem \ref{hatareset} for all remaining cases. Assume $d=3i+p+2$, where $i> p+1$ (the case $i=p+1$ corresponds to $d=2g+1$, which is our hypothesis).

If $L\in \mbox{Pic}^d(C)$ satisfies $K_{p,2}(C,L)\neq 0$, then by Lemma \ref{reduction-by-one}, we obtain $K_{i-1,2}(C,L(D))\neq 0$, for each divisor $D\in C_{i-p-1}$. Since  $\mbox{deg}(L(D))=4i+1$ we can apply Theorem \ref{even} in that case and conclude that $L-K_C+C_{i-p-1}\subset C_{i+1}-C_{i-2}$, that is, $\mbox{dim } V_{g-p-3}^{g-p-4}(2K_C-L)\geq i-p-1$, which again is assumed not to happen.
\hfill $\Box$

\vskip 4pt

Theorem \ref{even}, that is, the situation when $g$ is even and $d=2g+1$ can be viewed as a limit case of the divisorial case dealt with by Theorem \ref{diveset}.

\vskip 3pt

\noindent \emph{Proof of Theorem \ref{even}.} Let $C$ be a general curve of genus $g=2i$ and  $L\in \mbox{Pic}^{4i+1}(C)$ a line bundle with $K_{i-1,2}(C,L)\neq 0$. Write $L=\OO_C(x_1+\cdots+x_{4i+1})$, for distinct points $x_j\in C$, and aim to show that
 $$L-K_C\in C_{i+1}-C_{i-2}.$$
We fix general points $x,y\in C$ and let $X:=C\cup E$ be the nodal curve of genus $2i+1$ obtained by attaching to $C$ at the points $x$ and $y$ a rational curve $E$. Since $C$ has only finitely many pencils $\mathfrak{g}^1_{i+1}$ and $x$ and $y$ are chosen generically, $\mbox{gon}(X)=i+2$. On $X$, we consider the line bundle $L_X$ of bidegree $(4i+1,1)$, having restrictions $L_{X|C}=L$ and $L_E=\OO_E(1)$ respectively. Choosing a point $x_{4i+2}\in E-\{x,y\}$, note that
$[X, L_X]$ corresponds to the $(2g+2)$-pointed stable curve $[X, x_1, \ldots, x_{4i+2}]\in \mm_{g+1,2g+2}$ under the Abel-Jacobi map $\mm_{g+1,2g+2}\dashrightarrow \overline{\mathfrak{Pic}}^{2g+2}_{g+1}$.

\vskip 4pt

We show that $K_{i-1,2}(X, L_X)\neq 0$. To that end, we write the duality theorem \cite{green-koszul}
$$K_{i-1,2}(X,L_X)^{\vee}\cong K_{i+1,0}(X,\omega_X,L_X).$$
Using the isomorphisms $H^0(X,L_X)\cong H^0(C,L)$ and $H^0(X, \omega_X)\cong H^0(C,K_C(x+y))$ respectively, we write the following commutative diagram, where the vertical arrows are both injective:
$$
\begin{CD}
{\bigwedge^{i+1} H^0(C,L)\otimes H^0(C,K_C)} @>{d_{i+1,0}}>>{\bigwedge^{i} H^0(C,L)\otimes H^0(C, L\otimes K_C) } \\
@V{}VV @V{}VV \\
{\bigwedge^{i+1} H^0(X,L_X)\otimes H^0(X,\omega_X)} @>{d_{i+1,0}}>> \bigwedge^i H^0(X,L_X)\otimes H^0(X, L_X\otimes \omega_X) \\
\end{CD}
$$
This yields an injection $0\neq K_{i+1,0}(C,K_C,L)\hookrightarrow K_{i+1,0}(X,\omega_X,L_X)$. This implies the non-vanishing $K_{i-1,2}(X,L_X)\neq 0$, that is,
$[X, x_1, \ldots, x_{2g+2}]\in \overline{\mathfrak{Syz}}$, the closure being taken inside $\mm_{g+1,2g+2}$. Recall that we have established the following equality of closed sets  in $\mm_{g+1,2g+2}$:
$$\overline{\mathfrak{Syz}}=\overline{\mathfrak{Sec}}\cup \overline{\mathfrak{Hur}}.$$
As already pointed out, $X$ has maximal gonality, that is, $[X, x_1, \ldots, x_{2g+2}]\notin \overline{\mathfrak{Hur}}$, therefore necessarily $[X, x_1, \ldots, x_{2g+2}]\in \overline{\mathfrak{Sec}}$. Rather than trying to understand the meaning of the secant condition on the singular curve $X$, we recall the identification \cite{farkas-popa-mustata}
$$Y_{i+1}-Y_{i-1}=\Theta_{\bigwedge ^{i-1} Q_Y}\subset \mbox{Pic}^2(Y),$$
valid for any \emph{smooth} curve $Y$ of genus $2i+1$. The right-hand-side of this equality makes sense for \emph{semi-stable} curves as well, in particular for $X$. Thus $L_X\otimes \omega_X^{\vee} \in \Theta_{\bigwedge^{i-1} Q_{\omega_X}}$, and by duality,
\begin{equation}\label{cond3}
H^0\Bigl(X,\bigwedge^{i-1} M_{\omega_X}\otimes \omega_X^{\otimes 2}\otimes L_X^{\vee}\Bigr)\neq 0.
\end{equation}
Observe that $M_{\omega_X |C}=M_{K_C(x+y)}$ and $M_{\omega_X|E}=\OO_E^{\oplus g}$, hence we have the Mayer-Vietoris sequence
$$0\longrightarrow \bigwedge ^{i-1} M_{\omega_X}\otimes \omega_X^{\otimes 2}\otimes L_X^{\vee}\longrightarrow \Bigl(\bigwedge^{i-1} M_{K_C(x+y)}\otimes \bigl(2K_C-L+2x+2y\bigr)\Bigr)\bigoplus \Bigl(\OO_E(-1)^{\oplus {g\choose i-1}}\Bigr)$$
$$\longrightarrow \Bigl(\bigwedge ^{i-1} M_{\omega_X}\otimes \omega_X^{\otimes 2}\otimes L_X^{\vee} \Bigr)_{\bigl |x,y}\rightarrow 0. $$
Condition (\ref{cond3}) yields $H^0\Bigl(C, \bigwedge^{i-1} M_{K_C}\otimes (2K_C-L+x+y)\Bigr)\neq 0$. Using again \cite{farkas-popa-mustata}, this non-vanishing translates into the condition
\begin{equation}\label{cond4}
L-K_C-x-y\in \bigl(C_i-C_{i-1} \bigr) + \bigl(C_{i+1}-C_{i-2}-x-y\bigr)\subset \mbox{Pic}^1(C),
\end{equation}
which holds for \emph{arbitrary} points $x,y\in C$. If, for some $x,y\in C$, the  bundle $L-K_C-x-y$ belongs to the second component of the cycle in (\ref{cond4}), we are finished. Else, if
$L-K_C-C_2\subset C_i-C_{i-1}$, then we find that $\mbox{dim } V_{i+2}^{i+1}(L)\geq 2$. Applying \cite[Lemma 4]{FHL}, we obtain that $V_{i+1}^i(L)\neq \emptyset$, completing the proof.
\hfill $\Box$


\begin{thebibliography}{aaaaaa}

\bibitem[AF1]{aprodu-farkas-green} M. Aprodu and G. Farkas, {\emph{The Green Conjecture for smooth curves lying on arbitrary $K3$ surfaces}}, Compositio Mathematica \textbf{147} (2011), 839-851.
\bibitem[AF2]{aprodu-farkas-covers} M. Aprodu and G. Farkas, {\emph{Green's Conjecture for general covers}}, Compact Moduli Spaces and Vector Bundles (V. Alexeev et al. ed.), Contemporary Mathematics Vol. 564, 2012, 211-226.
\bibitem[AC]{AC} E. Arbarello and M. Cornalba, {\em{The Picard group of the moduli space of curves}}, Topology \textbf{26} (1987), 153-171.    
\bibitem[ACGH]{ACGH} E. Arbarello, M. Cornalba, P. A. Griffiths and J. Harris, {\emph{Geometry of algebraic curves}}, Volume I, Grundlehren der mathematischen Wissenschaften \textbf{267},
Springer-Verlag (1985).
\bibitem[CEFS]{chiodo-eisenbud-farkas-schreyer} A. Chiodo, D. Eisenbud, G. Farkas and F.-O. Schreyer, {\em{Syzygies of torsion bundles and the geometry of the level $\ell$ modular varieties over $\overline{\mathcal{M}}_g$}}, Inventiones Math. \textbf{194} (2013), 73-118.
\bibitem[Dol]{dolgachev} I. Dolgachev, {\em{Mirror symmetry for lattice polarized $K3$ surfaces}}, Journal of Mathematical Sciences \textbf{81} (1996), 2599-2630.
\bibitem[Fa1]{Fa} G. Farkas, {\em{Koszul divisors on moduli spaces of curves}}, American Journal of Mathematics \textbf{131} (2009), 819-869.
\bibitem[Fa2]{Fa2} G. Farkas, {\em{Brill-Noether with ramification at unassigned points}}, Journal of Pure and Applied Algebra \textbf{217} (2013), 1838-1843.
\bibitem[FMP]{farkas-popa-mustata} G. Farkas, M. Musta\c{t}\u{a} and M. Popa, {\em{Divisors on
$\cM_{g, g+1}$ and the Minimal Resolution Conjecture for points on
canonical curves}}, Annales Sci. de L'\'Ecole  Normale Sup\'erieure \textbf{36} (2003), 553-581.
\bibitem[FL]{FL} G. Farkas and K. Ludwig, {\em{The Kodaira dimension of the moduli
space of Prym varieties}}, Journal of the European Mathematical Society \textbf{12} (2010), 755-795.
\bibitem[FV]{farkas-verra-moduli-theta} G. Farkas and A. Verra, {\em{Moduli of theta-characteristics via Nikulin surfaces}}, Mathematische Annalen \textbf{354} (2012), 465-496.
\bibitem[FHL]{FHL} W. Fulton, J. Harris and R. Lazarsfeld, {\em{Excess linear series on an algebraic curve}}, Proceedings of the American Mathematical Society, \textbf{92} (1984), 320-322.
\bibitem[GS]{gar-sarti-even} A. Garbagnati and A. Sarti, {\em{Projective models of $K3$ surfaces with an even set}}, Advances in Geometry \textbf{8} (2008), 413-440.
\bibitem[vGS]{vGS} B. van Geemen and A. Sarti, {\em{Nikulin involutions on $K3$ surfaces}}, Mathematische Zeitschrift \textbf{255} (2007), 731-753.
\bibitem[G]{green-koszul} M. Green, {\em{Koszul cohomology and the cohomology of projective varieties}}, Journal of Differential Geometry \textbf{19} (1984), 125-171.
\bibitem[GL1]{green-lazarsfeld-projective} M. Green and R. Lazarsfeld, {\em{On the projective normality of complete linear series on an algebraic curve}}, Inventiones Math. \textbf{83} (1986), 73-90.
\bibitem[GL2]{green-lazarsfeld-special} M. Green and R. Lazarsfeld, {\em{Special divisors on curves on a $K3$ surface}}, Inventiones Math. \textbf{89} (1987), 357-370.
\bibitem[GL3]{GL3} M. Green and R. Lazarsfeld, {\em{Some results on the syzygies of finite sets and algebraic curves}}, Compositio Mathematica \textbf{67} (1988), 301-314.
\bibitem[HM]{HM} J. Harris and D. Mumford, {\em{On the Kodaira dimension of $\mm_g$}}, Inventiones Math. \textbf{67} (1982), 23-88.
\bibitem[H]{huy-lec-k3} D. Huybrechts, {\em{Lectures on $K3$ surfaces}}, notes available at www.math.uni-bonn.de/people/huybrech/ K3Global.pdf.
\bibitem[K]{kemeny-singular} M. Kemeny, {\em{The Moduli of Singular Curves on K3 Surfaces}}, arXiv:1401.1047.
\bibitem[KS]{koh-stillman} J. Koh and M. Stillman, {\em{Linear syzygies and line bundles on an algebraic curve}}, Journal of Algebra \textbf{125} (1989), 120-132.
\bibitem[La1]{lazarsfeld-bnp} R. Lazarsfeld, {\em{Brill-Noether-Petri without degenerations}}, Journal of Differential Geometry
\textbf{23} (1986), 299-307.
\bibitem[La2]{lazarsfeld-kernel} R. Lazarsfeld, {\em{A sampling of vector bundle techniques in the study of linear series}}, Lectures on Riemann surfaces
(Trieste, 1987), 500-559.
\bibitem[M]{mayer-families} A. Mayer, {\em{Families of $K3$ surfaces}}, Nagoya Mathematical Journal \textbf{48} (1972), 1-17.
\bibitem[SD]{donat} B. Saint-Donat, {\em{Projective models of $K3$ surfaces}}, American Journal of Mathematics \textbf{96} (1974), 602-639.
\bibitem[Sch]{Sch} F.-O. Schreyer, {\em{Syzygies of canonical curves and special linear series}}, Mathematische Annalen \textbf{275} (1986), 105-137.
\bibitem[V1]{voisin-even} C. Voisin, {\em{Green's generic syzygy conjecture
for curves of even genus lying on a $K3$ surface}}, Journal of
European Mathematical Society \textbf{4} (2002), 363-404.
\bibitem[V2]{voisin-odd} C. Voisin, {\em{Green's canonical syzygy conjecture for generic curves of odd genus}},
Compositio Mathematica \textbf{141} (2005), 1163--1190.
\end{thebibliography}
\end{document}